\documentclass[11pt,reqno]{amsart}
\usepackage[T1]{fontenc}
\usepackage[utf8]{inputenc}
\usepackage{lmodern,amsmath,amssymb,amsthm,booktabs,microtype}
\usepackage[hidelinks]{hyperref}
\hypersetup{
  pdftitle={Strict correlation-dimension drop for quartic Salem Bernoulli convolutions},
  pdfsubject={Correlation dimension, exact collisions, and rigorous explicit bounds},
  pdfkeywords={Bernoulli convolution, Salem number, correlation dimension, exact collision, toral automorphism, computer-assisted proof}
}
\numberwithin{equation}{section}

\newtheorem{theorem}{Theorem}[section]
\newtheorem{proposition}[theorem]{Proposition}
\newtheorem{lemma}[theorem]{Lemma}
\newtheorem{corollary}[theorem]{Corollary}
\theoremstyle{definition}
\newtheorem{definition}[theorem]{Definition}

\newcommand{\RR}{\mathbb R}
\newcommand{\ZZ}{\mathbb Z}
\newcommand{\QQ}{\mathbb Q}
\newcommand{\CC}{\mathbb C}
\newcommand{\TT}{\mathbb T}
\newcommand{\PP}{\mathbb P}
\newcommand{\Shift}{\mathcal S}
\DeclareMathOperator{\Law}{Law}
\DeclareMathOperator{\can}{can}
\DeclareMathOperator{\dist}{dist}
\DeclareMathOperator{\supp}{supp}

\makeatletter
\newcommand{\sectionnotoc}[1]{%
  \let\saved@tocwrite\@tocwrite
  \let\@tocwrite\@gobbletwo
  \section*{#1}%
  \let\@tocwrite\saved@tocwrite
}
\makeatother

\title[Quartic Salem Bernoulli convolutions]{Strict correlation-dimension drop for quartic Salem Bernoulli convolutions}
\date{}

\begin{document}

\author{Guozheng Cheng}
\address{School of Mathematical Sciences, Dalian University of Technology, Dalian, 116024, P. R. China}
\email{gzhcheng@dlut.edu.cn}

\author{Xiang Fang}
\address{Department of Applied Mathematics, National Yang Ming Chiao Tung University, Hsinchu, 300093, Taiwan}
\email{xfang@nycu.edu.tw}

\author{Xueqing Ma}
\address{School of Mathematical Sciences, Dalian University of Technology, Dalian, 116024, P. R. China}
\email{maxueqing4680@mail.dlut.edu.cn}

\author{Hongli Zhang}
\address{School of Mathematical Sciences, Dalian University of Technology, Dalian, 116024, P. R. China}
\email{zhanghongli1@mail.dlut.edu.cn}

\begin{abstract}
For every quartic Salem number $\beta\in(1,2)$, we prove that the Bernoulli convolution $\nu_{\beta^{-1}}$ with equal weights has correlation dimension strictly less than one. In particular, neither of the two corresponding measures has an $L^2$ density. We relate correlation dimension to exact collision probabilities and express these probabilities as positive integrals over a four-dimensional torus associated with the defining polynomial. The unit-circle conjugates produce a persistent oscillatory term; using carefully chosen return times of the corresponding rotation, we take differences that make this term small while controlling the length and cost of the resulting configurations. Many well-separated placements of these configurations then give a collision probability above the critical scale, which yields the strict dimension drop. The qualitative argument is independent of the computer-assisted estimates. Separate computer-assisted arguments give
\[
 0.99999<D_2(\nu_{\beta_1^{-1}})<1-10^{-65},
 \qquad
 0.99999<D_2(\nu_{\beta_2^{-1}})<1-10^{-75},
\]
where $\beta_1<\beta_2$ are the two quartic Salem numbers in $(1,2)$.
\end{abstract}
\maketitle

\vspace{-2em}

\tableofcontents

\section{Introduction}\label{sec:intro}

For $0<r<1$, let
\[
\nu_r=\Law\left(\sum_{j=0}^{\infty}\varepsilon_j r^j\right),
 \qquad
 \PP(\varepsilon_j=1)=\PP(\varepsilon_j=-1)=\frac12,
\]
where the signs are independent. Bernoulli convolutions form a basic family of self-similar measures whose regularity can depend strongly on the arithmetic of the contraction parameter. In the overlapping range $r>1/2$, the support is an interval, but this alone gives little information about how regular the measure is. The surveys \cite{PSS,Varju,VarjuICM} describe the main dimension and regularity questions for this family.

We study the correlation dimension of these measures. For a compactly supported Borel probability measure $\mu$ on $\RR$, put
\[
 I_s(\mu)=\iint_{\RR^2}|x-y|^{-s}\,d\mu(x)\,d\mu(y),
 \qquad 0<s<1,
\]
and define
\[
 D_2(\mu)=\sup\bigl(\{0\}\cup\{s\in(0,1):I_s(\mu)<\infty\}\bigr).
\]
We use this energy formulation of the order-two, or correlation, dimension throughout; see \cite[Definitions~2.2--2.3]{KPV}. It measures how often two independent samples from $\mu$ are close to one another at small scales. In particular, a compactly supported probability measure with an $L^2(\RR)$ density has $D_2=1$.

A Salem number $\beta>1$ is a real algebraic integer whose conjugates consist of $\beta$, $\beta^{-1}$, and one or more pairs on the unit circle. We consider $1<\beta<2$ and set $r=\beta^{-1}$. The unit-circle conjugates create a difficulty that is absent for Pisot numbers. For a Pisot number, all nonprincipal conjugates lie strictly inside the unit circle, so their contributions decay exponentially. For a Salem number, the unit-circle conjugates instead produce an oscillation that does not decay. In degree four there is exactly one such conjugate pair. Since four is the smallest possible degree of a Salem number, this is the first case in which this difficulty appears.

There are precisely two quartic Salem numbers in $(1,2)$, as proved in Proposition~\ref{prop:quartic-class}. We denote them by $\beta_1<\beta_2$; they are the roots greater than one of
\begin{equation}\label{eq:quartic-polynomials}
 P_1(X)=X^4-X^3-X^2-X+1,
 \qquad
 P_2(X)=X^4-2X^3+X^2-2X+1.
\end{equation}
Thus the following theorem covers the complete quartic class in the overlapping range.

\begin{theorem}[Main theorem]\label{thm:main}
Let $\beta\in(1,2)$ be a quartic Salem number. Then
\[
 D_2(\nu_{\beta^{-1}})<1.
\]
\end{theorem}

It follows that neither $\nu_{\beta_1^{-1}}$ nor $\nu_{\beta_2^{-1}}$ has an $L^2$ density. The theorem concerns correlation dimension: by itself it does not imply that either measure is singular, nor does it imply a drop in Hausdorff dimension. The proof of Theorem~\ref{thm:main} is independent of the computer-assisted estimates given in the appendices.

\subsection{Previous results and explicit bounds}

In the Pisot case, Garsia's entropy estimate gives Hausdorff dimension strictly smaller than one, and hence also $D_2<1$; see \cite{GarsiaEntropy} and \cite[\S4]{Varju}. For Salem parameters the situation is different. Feng proved a nontrivial range of local dimensions for a Salem family containing the first parameter in \eqref{eq:quartic-polynomials} \cite[Theorem~1.2]{Feng}. This gives detailed information about the local behavior of the measure, but it does not decide whether its correlation dimension is equal to one. Theorem~\ref{thm:main} settles this order-two question for both quartic Salem parameters.

The correlation-dimension problem is related to, but distinct from, the general dimension theory of Bernoulli convolutions. Varj\'u's ICM report \cite[\S2, Theorem~8]{VarjuICM} describes the entropy-rate formula for the Hausdorff dimension of algebraic Bernoulli convolutions, including systems with exact overlaps. Here correlation dimension is governed by squared atom masses in the finite approximations. Shmerkin's $L^q$-dimension formula under exponential separation \cite[Theorem~1.3]{ShmerkinSurvey} provides a natural comparison. Here we estimate the order-two collision rate directly in the presence of exact coincidences.

Fourier decay gives complementary information about Salem Bernoulli convolutions. Marshall-Maldonado and Solomyak \cite[Corollary~B.1]{MMS2026} prove a quantitative Fourier-decay estimate of $\log^*$ type. This concerns pointwise decay of the Fourier transform, whereas Theorem~\ref{thm:main} gives an obstruction to $L^2$ regularity through pair concentration. The two results address different aspects of the regularity problem.

During the final preparation stage of this paper, OpenAI released a manuscript \cite{OpenAI} proving that $\nu_{\beta^{-1}}$ is singular for every quartic Salem number $\beta\in(1,2)$. This settles the measure-type question for the parameters considered here, but does not imply the strict correlation-dimension drop in Theorem~\ref{thm:main}. Their argument is different from the collision-pressure approach developed in this paper.

For the two quartic parameters, Kleptsyn, Pollicott and Vytnova obtained the rigorous lower bounds
\[
 D_2(\nu_{\beta_1^{-1}})>0.997968750,
 \qquad
 D_2(\nu_{\beta_2^{-1}})>0.990312500
\]
\cite[Appendix~A.1]{KPV}. These estimates show that the correlation dimensions are bounded very close to one from below, but they are consistent with full correlation dimension. The appendices of the present paper give stronger lower bounds together with explicit upper bounds.

\begin{theorem}[Explicit bounds]\label{thm:quantitative}
For the two quartic Salem parameters,
\[
 0.99999<D_2(\nu_{\beta_1^{-1}})<1-10^{-65},
 \qquad
 0.99999<D_2(\nu_{\beta_2^{-1}})<1-10^{-75}.
\]
\end{theorem}

The quantities $10^{-65}$ and $10^{-75}$ are lower bounds for the deficits $1-D_2$, not estimates of their actual sizes. Thus Theorem~\ref{thm:quantitative} gives rigorous enclosures rather than high-precision determinations of the two dimensions. The upper and lower bounds are proved by different finite arguments: the upper bounds use explicit packets and prescribed return times, while the lower bounds use finite supersolutions for the complete return process. Appendix~\ref{app:verification} gives the exact arithmetic and verification details.

\subsection{Collision pressure and packet amplification}

Fix one of the two parameters $\beta$, write $r=\beta^{-1}$, and let $P$ be its minimal polynomial. Let $\zeta$ be the unit-circle conjugate of $\beta$ with positive imaginary part. Let
\[
 X_n=\sum_{j=0}^{n-1}\varepsilon_j r^j,
 \qquad
 Q_n=\PP(X_n=X_n'),
\]
where $X_n'$ is an independent copy of $X_n$. If $p_n(x)=\PP(X_n=x)$, then
\[
 Q_n=\sum_x p_n(x)^2.
\]
Thus $Q_n$ measures the quadratic concentration of the $n$th-level approximation to $\nu_r$.

The same quantity has a useful arithmetic interpretation. Reversing the half-difference digits gives
\begin{equation}\label{eq:lazy-return}
\begin{aligned}
 Q_n&=\PP\left(\sum_{j=0}^{n-1}U_j\beta^j=0\right),\\
 \PP(U_j=0)&=\frac12,
 \qquad \PP(U_j=1)=\PP(U_j=-1)=\frac14.
\end{aligned}
\end{equation}
Here the variables $U_j$ are independent. Thus $Q_n$ is an exact return probability for a lazy random walk whose step sizes grow like powers of $\beta$. A related random-walk viewpoint with exponentially increasing step lengths appears in \cite{Neunhaeuserer}.

Exact collisions on two consecutive blocks can be imposed independently, so
\[
 Q_{n+m}\ge Q_nQ_m.
\]
It follows that the exponential collision rate
\[
 \varrho=\lim_{n\to\infty}Q_n^{1/n}
\]
exists. Section~\ref{sec:collisions} proves
\[
 D_2(\nu_{\beta^{-1}})
 =-\frac{\log\varrho}{\log\beta}.
\]
Consequently,
\begin{equation}\label{eq:collision-criterion}
 D_2(\nu_{\beta^{-1}})<1
 \quad\Longleftrightarrow\quad
 \beta^nQ_n>1
 \quad\text{for some }n\ge1.
\end{equation}

This is the quantitative target of the proof. The existence of exact collisions is not enough: at some scale their total probability must exceed the critical value $\beta^{-n}$. Different symbolic words may already represent the same atom in $p_n$; the separation estimate used later concerns distinct atoms after these coincidences have been identified, not exponential separation of all symbolic words.

The next step is to represent the collision probability as a positive integral over a torus. Let $T_P$ be the toral automorphism associated with the defining polynomial $P$, and let $y_j(x)$ denote the first coordinate, represented in $[-1/2,1/2)$, of the $j$th iterate of $x\in\TT^4$. Section~\ref{sec:torus} proves
\[
 Q_n=\int_{\TT^4}\prod_{j=0}^{n-1}\cos^2\bigl(\pi y_j(x)\bigr)\,dx.
\]
Here $dx$ is Haar probability measure on $\TT^4$. Indeed, after expanding the cosine factors, integration over the torus retains exactly the terms corresponding to the collision relations. The connection between Salem arithmetic, toral automorphisms, and symbolic representations already appears in the work of Lindenstrauss and Schmidt \cite{LindenstraussSchmidt}.

The positivity of the integrand is important for the lower bound. A suitable recurrence configuration gives a small region of the torus on which the product can be estimated from below. We later arrange many such regions to be disjoint, so their contributions can be added without cancellation. The remaining task is therefore to construct enough recurrence configurations for which these lower bounds remain sufficiently large.

We call a suitable recurrence configuration a \emph{packet}. It consists of a real bilateral sequence $Y$ for which the recurrence defined by $P$ produces a nonzero integer-valued sequence supported on a finite interval, called the core. Away from the core, $Y$ has geometrically decaying real tails together with one outgoing oscillatory term coming from the unit-circle conjugates. Write $\ell$ for the core length and translate the core to $[0,\ell-1]$. The outgoing term is $2\Re(\gamma\zeta^{j-\ell})$ for $j\ge\ell$, with $\gamma\ne0$, and is absent for $j<\ell$. Removing this term gives the localized sequence $H$, whose coordinates lie in $(-1/2,1/2)$ and whose tails decay geometrically on both sides. We write
\[
 \Delta=2|\gamma|,
 \qquad
 E=\sum_{j \in \mathbb{Z}} -\log\cos^2(\pi H_j).
\]
Thus $\Delta$ measures the size of the outgoing oscillation. The action $E$ is the sum of the logarithmic losses of the cosine factors for $H$; it measures the cost of the localized configuration in the torus integral.

The amplification step uses many separated copies of the same packet. Suppose that the localized configurations stay a common positive distance from the endpoints $\pm1/2$, where the cosine factor vanishes, and satisfy a common geometric bound on their tails. If a packet has neutral amplitude $\Delta$, we place its copies at separations of order
\[
 N\asymp\Delta^{-1}.
\]
For $N>2\ell+4$, allow each successive separation to range over the integers from $N$ to $2N-1$. At every step, a fixed positive proportion of these choices keeps the combined outgoing oscillation bounded by a constant multiple of $\Delta$, independently of the number of copies. Thus each added copy gives a logarithmic gain of order
\[
 \log N\asymp\log(1/\Delta).
\]
Its principal cost is the action $E$. In addition, the small residual oscillation perturbs a core of length $\ell$, and this error is negligible when $\Delta\ell$ is small.

Corollary~\ref{cor:packet-criterion} therefore reduces the proof, in the quartic setting, to constructing packets with the uniform bounds above such that
\begin{equation}\label{eq:intro-packet-criterion}
 \Delta_k\longrightarrow0,
 \qquad
 \Delta_k\ell_k\longrightarrow0,
 \qquad
 E_k=o\bigl(\log(1/\Delta_k)\bigr).
\end{equation}
The first condition provides increasingly many possible placements, the second controls the accumulated perturbation on the growing core, and the third ensures that the cost of each packet is smaller than the logarithmic gain from those placements. Once these conditions hold, the amplification argument gives $\beta^nQ_n>1$ for some $n$, and hence $D_2(\nu_{\beta^{-1}})<1$ by \eqref{eq:collision-criterion}.

We construct the required packets by taking reciprocal differences at carefully chosen return times of the rotation coming from the unit-circle conjugates. An algebraic construction first provides an initial packet whose localized coordinates stay a positive distance from $\pm1/2$. For the reciprocal packets constructed in Section~\ref{sec:reciprocal}, the step in Lemma~\ref{lem:mirror} uses a suitably chosen even delay $L>\ell$ and changes the outgoing amplitude to
\[
 \Delta^{+}=\Delta|1-\zeta^L|.
\]
A better return makes the factor $|1-\zeta^L|$ smaller, but a longer delay also makes the new packet longer and can increase its action. The chosen delays must separate the two copies of the old core while keeping the length and action under control. The controlled iteration does so while preserving the common boundary margin and tail bounds: it gives $E_k=O(2^k)$, makes $\log(1/\Delta_k)/2^k$ tend to infinity, and ensures $\Delta_k\ell_k\to0$. Hence the three conditions in \eqref{eq:intro-packet-criterion} are satisfied. The return times are obtained from finite best-return minimization and the pigeonhole principle; no bounded-type assumption or quantitative irrationality exponent is imposed on the rotation.

\subsection{Scope and organization}

The argument developed here treats the complete quartic Salem class in the overlapping range. Its main feature is the control of a single pair of unit-circle conjugates. We do not claim a corresponding result in higher Salem degree, where several oscillatory pairs would have to be controlled simultaneously.

Section~\ref{sec:collisions} relates correlation dimension to exact collision probabilities. Section~\ref{sec:torus} gives the positive torus representation and the amplification criterion. Sections~\ref{sec:reciprocal} and \ref{sec:returns} construct the reciprocal packets and prove Theorem~\ref{thm:main}. The appendices give the explicit upper and lower bounds in Theorem~\ref{thm:quantitative}, together with the exact verification details.

\section{Algebraic preliminaries and collision probabilities}\label{sec:collisions}
Fix one of the two quartic polynomials in \eqref{eq:quartic-polynomials} and, when no confusion can arise, write
\begin{equation}\label{eq:general-quartic}
 P(X)=X^4-AX^3+BX^2-AX+1,
\end{equation}
where $(A,B)=(1,-1)$ or $(2,1)$.  We write $\Shift$ for the bilateral shift, $(\Shift y)_j=y_{j+1}$, so that
\begin{equation}\label{eq:recurrence}
 (P(\Shift)y)_j=y_j-Ay_{j+1}+By_{j+2}-Ay_{j+3}+y_{j+4}.
\end{equation}

\begin{proposition}\label{prop:quartic-class}
The monic irreducible reciprocal quartics having a Salem root in $(1,2)$ are precisely $P_1$ and $P_2$.  For $i=1,2$, write their roots as
\[
 \beta_i,
 \qquad r_i=\beta_i^{-1},
 \qquad \zeta_i,
 \qquad \bar\zeta_i,
 \quad \Im\zeta_i>0.
\]
When one parameter is fixed, we suppress the index and write $P,\beta,r,\zeta$.  These roots satisfy
\begin{equation}\label{eq:root-bounds}
 \frac53<\beta<2,
 \qquad \frac12<r<\frac35,
 \qquad |\zeta|=1.
\end{equation}
Moreover,
\begin{equation}\label{eq:unit-circle-bounds}
 |1-\zeta|>\frac32,
 \qquad |1-\zeta^2|>\frac32,
 \qquad |P'(\zeta)|>1,
\end{equation}
and neither $\zeta$ nor $\zeta^2$ is a root of unity.
\end{proposition}
\begin{proof}
Let \eqref{eq:general-quartic} be the minimal polynomial of a quartic Salem number $\beta\in(1,2)$ and put
\[
 u=\beta+\beta^{-1},\qquad v=\zeta+\zeta^{-1}.
\]
Then $2<u<5/2$, $-2<v<2$, and the trace polynomial is
$T^2-AT+B-2$.  Hence $A=u+v\in\{1,2,3,4\}$ and
$B=Au-u^2+2$.  As $u$ ranges over $(2,5/2)$ the possible ranges of the latter expression are
\[
\begin{array}{c|cccc}
 A&1&2&3&4\\ \hline
 Au-u^2+2&(-7/4,0)&(3/4,2)&(13/4,4)&(23/4,6).
\end{array}
\]
Only $(A,B)=(1,-1)$ and $(2,1)$ give an integral value of $B$.

Conversely,
\[
 P_i(X)=X^2F_i(X+X^{-1}),
 \qquad
 F_1(T)=T^2-T-3,
 \quad
 F_2(T)=T^2-2T-1.
\]
The two roots of the trace polynomial are
\begin{equation}\label{eq:trace-roots}
\begin{aligned}
 u_1&=\frac{1+\sqrt{13}}2,&v_1&=\frac{1-\sqrt{13}}2,\\
 u_2&=1+\sqrt2,&v_2&=1-\sqrt2.
\end{aligned}
\end{equation}
The inequalities $2<u_i<5/2$ and $-2<v_i<2$ give one real reciprocal pair and one nonreal unit-circle pair.  In $K_0=\QQ(u_i)$, the other real embedding sends $u_i^2-4$ to $v_i^2-4<0$.  Thus $u_i^2-4$ is not a square in $K_0$, while
$X^2-u_iX+1$ has roots $\beta_i,\beta_i^{-1}$.  It follows that $[\QQ(\beta_i):\QQ]=4$, proving irreducibility.

The values $u_i$ in \eqref{eq:trace-roots} exceed $34/15=5/3+3/5$.  Since $x+x^{-1}$ is increasing on $(1,\infty)$, this gives \eqref{eq:root-bounds}.  They also give $v_i<-1/4$ and $v_i^2<7/4$, and therefore
\[
 |1-\zeta_i|^2=2-v_i>\frac94,
 \qquad
 |1-\zeta_i^2|^2=4-v_i^2>\frac94.
\]
Differentiating the trace factorization at $\zeta_i$ yields
\[
 |P_i'(\zeta_i)|=(u_i-v_i)\sqrt{4-v_i^2}>1.
\]
Finally, if $\zeta_i$ were a root of unity, irreducibility would force all conjugates of $\zeta_i$ to have modulus one, contrary to $\beta_i>1$; the same applies to $\zeta_i^2$.
\end{proof}

\begin{lemma}\label{lem:homogeneous-modes}
Every real bilateral solution of $P(\Shift)y=0$ is a real linear combination of
$\beta^j$, $r^j$, $\cos(j\varphi)$ and $\sin(j\varphi)$, where
$\zeta=e^{i\varphi}$.  If such a solution tends to zero as $j\to-\infty$ and is bounded as $j\to+\infty$, then it vanishes identically.
\end{lemma}
\begin{proof}
The four displayed modes are linearly independent by the Vandermonde determinant, and four consecutive coordinates determine a bilateral solution.  Hence they span the real solution space.  Decay at $-\infty$ eliminates the $r^j$ term.  Since $\sin\varphi\ne0$, two consecutive oscillatory observations recover the cosine and sine coefficients through a fixed invertible $2\times2$ matrix; their convergence to zero at $-\infty$ therefore eliminates both oscillatory terms.  Boundedness at $+\infty$ then eliminates the $\beta^j$ term.
\end{proof}

Let
\[
 X_n=\sum_{j=0}^{n-1}\varepsilon_jr^j,
\]
let $\mathcal A_n$ be its support, and put $p_n(x)=\PP(X_n=x)$.  If $X_n'$ is an independent copy, then
\begin{equation}\label{eq:collision-squares}
 Q_n=\PP(X_n=X_n')=\sum_{x\in\mathcal A_n}p_n(x)^2.
\end{equation}
Collisions on two consecutive blocks are independent, so
$Q_{n+m}\ge Q_nQ_m$.  Fekete's lemma gives the collision growth rate
\begin{equation}\label{eq:collision-rate}
 \varrho=\lim_{n\to\infty}Q_n^{1/n}
 =\sup_{n\ge1}Q_n^{1/n}.
\end{equation}
We also put $Q_0=1$ and
\[
 G(t)=\sum_{n=0}^{\infty}t^nQ_n.
\]

\begin{lemma}\label{lem:salem-separation}
For distinct $x,y\in\mathcal A_n$,
\begin{equation}\label{eq:atom-separation}
 |x-y|\ge 2(\beta-1)r^nn^{-2}.
\end{equation}
It follows that $r\le\varrho<1$.  If $X,X'$ are independent with law $\nu_r$ and
$\mathcal C(t)=\PP(|X-X'|\le t)$, then there is a constant $C_{\rm sep}<\infty$ such that
\begin{equation}\label{eq:near-collisions}
 \mathcal C(r^n)\le C_{\rm sep}(1+n^2)Q_n
 \qquad(n\ge1).
\end{equation}
In particular, $\PP(X=X')=0$.
\end{lemma}
\begin{proof}
For distinct atoms, write
$x-y=2r^{n-1}F(\beta)$, where $F$ has degree at most $n-1$ and coefficients in $\{-1,0,1\}$.  Since $F(\beta)$ is a nonzero algebraic integer,
\[
 1\le |N_{\QQ(\beta)/\QQ}(F(\beta))|,
 \qquad
 |F(r)|\le(1-r)^{-1},
 \qquad
 |F(\zeta)|\le n.
\]
The same bound holds at $\bar\zeta$, so
$|F(\beta)|\ge(1-r)n^{-2}$, which gives \eqref{eq:atom-separation}.  All atoms lie in an interval of length $2/(1-r)$; hence
$|\mathcal A_n|\le 1+C r^{-n}n^2$ for a constant depending only on $\beta$.  By Cauchy--Schwarz and \eqref{eq:collision-squares},
$Q_n\ge|\mathcal A_n|^{-1}$, and therefore $\varrho\ge r$.

To see that $\varrho<1$, choose $k$ with $r^k<1/2$ and condition on all signs except those at the positions $0,k,2k,\ldots$.  Distinct choices of these remaining signs give distinct sums, because the first differing term dominates all later terms in this sparse subseries.  Hence every atom has mass at most $2^{-\lceil n/k\rceil}$, so \eqref{eq:collision-squares} gives the same upper bound for $Q_n$.  Thus $\varrho\le2^{-1/k}<1$.

Finally, $|X-X_n|\le r^n/(1-r)$.  A near collision $|X-X'|\le r^n$ therefore implies
$|X_n-X_n'|\le H_*r^n$, where $H_*=1+2/(1-r)$.  By \eqref{eq:atom-separation}, every atom has at most
$1+H_*n^2/(\beta-1)$ neighbors at this distance.  Summing over the corresponding symmetric neighbor relation and using
$p_n(x)p_n(y)\le\frac12(p_n(x)^2+p_n(y)^2)$ proves \eqref{eq:near-collisions}.  Its right-hand side tends to zero because $\varrho<1$, so $\PP(X=X')=0$.
\end{proof}

\begin{theorem}\label{thm:collision-pressure}
For either quartic parameter,
\begin{equation}\label{eq:dimension-pressure}
 D_2(\nu_r)=-\frac{\log\varrho}{\log\beta},
 \qquad r\le\varrho<1.
\end{equation}
In particular,
\begin{equation}\label{eq:collision-criterion-body}
 D_2(\nu_r)<1
 \quad\Longleftrightarrow\quad
 \varrho>r
 \quad\Longleftrightarrow\quad
 \beta^nQ_n>1\ \text{for some }n\ge1.
\end{equation}
\end{theorem}
\begin{proof}
Set $a=-\log\varrho/\log\beta\in(0,1]$.  If $0<s<a$, choose
$\varrho<\lambda<\beta^{-s}$.  
Then $Q_n\le\lambda^n$ for all sufficiently large $n$.  A decomposition
into the shells $r^{n+1}<|X-X'|\le r^n$ gives
\[
 I_s(\nu_r)
 \le 1+r^{-s}\mathcal C(1)
   +\sum_{n=1}^{\infty}r^{-s(n+1)}\mathcal C(r^n).
\]
The first two terms and any finite initial segment of the sum are harmless.
By \eqref{eq:near-collisions}, the tail is bounded by a constant multiple of
\[
 \sum_{n=1}^{\infty}(1+n^2)(\beta^s\lambda)^n,
\]
which converges since $\beta^s\lambda<1$.  Thus
$I_s(\nu_r)<\infty$ and hence $D_2(\nu_r)\ge a$.

Now suppose $a<s<1$.  By \eqref{eq:collision-rate}, some $k\ge1$ satisfies
$\beta^{sk}Q_k>1$.  Require $m$ independent consecutive blocks of length $k$ to collide.  This event has probability $Q_k^m$, and on it
$|X-X'|\le2r^{mk}/(1-r)$.  Hence
\[
 I_s(\nu_r)
 \ge \left(\frac{1-r}{2}\right)^s
      (\beta^{sk}Q_k)^m
 \longrightarrow\infty.
\]
Thus $D_2(\nu_r)\le a$ when $a<1$; when $a=1$ the same inequality follows from the definition.  This proves \eqref{eq:dimension-pressure}.  The equivalences in \eqref{eq:collision-criterion-body} follow from \eqref{eq:collision-rate}.
\end{proof}

\begin{corollary}\label{cor:finite-collision}
If $0<s<1$ and $\beta^{sn}Q_n>1$ for some $n\ge1$, then
\begin{equation}\label{eq:strict-finite-collision}
 D_2(\nu_r)
 \le -\frac{\log Q_n}{n\log\beta}
 <s,
 \qquad
 I_s(\nu_r)=\infty.
\end{equation}
If $1<t<\beta$ and $G(t)<\infty$, then
\begin{equation}\label{eq:lower-from-G}
 D_2(\nu_r)\ge\frac{\log t}{\log\beta},
\end{equation}
and $I_s(\nu_r)<\infty$ for every $0<s<\log t/\log\beta$.
\end{corollary}
\begin{proof}
Since $\varrho\ge Q_n^{1/n}$, the first inequality in \eqref{eq:strict-finite-collision} follows from \eqref{eq:dimension-pressure}, and the assumed finite collision inequality makes it strict.  The repeated-block argument in the proof of Theorem~\ref{thm:collision-pressure} gives $I_s(\nu_r)=\infty$.

For the second assertion, positivity gives $Q_n\le G(t)t^{-n}$.  Substitution into \eqref{eq:near-collisions} makes the shell sum for $I_s(\nu_r)$ converge whenever $\beta^s<t$, giving \eqref{eq:lower-from-G}.
\end{proof}

When both parameters occur simultaneously, we write $Q_n^{(i)}$ and $G_i(t)$ for the collision probabilities and return series associated with $\beta_i$.

\section{A positive torus formula and packet amplification}\label{sec:torus}
Let $T_P$ be the automorphism of $\TT^4=\RR^4/\ZZ^4$ induced by
\begin{equation*}
 T_P(x_0,x_1,x_2,x_3)=(x_1,x_2,x_3,-x_0+Ax_1-Bx_2+Ax_3)\pmod{\ZZ^4}.
\end{equation*}
Its integer matrix $\mathsf T$ has determinant one. Write
\[
 \can(t)=t-\lfloor t+1/2\rfloor\in[-1/2,1/2),
\]
and let $y_j(x)$ be the canonical first coordinate of $T_P^jx$. We refer to $\ZZ+1/2$ as the centering boundary.

\begin{lemma}\label{lem:torus-formula}
With Haar probability measure on $\TT^4$,
\begin{equation}
 Q_n=\int_{\TT^4}\prod_{j=0}^{n-1}\cos^2(\pi y_j(x))\,dx.\label{eq:torus-formula}
\end{equation}
If a real bilateral sequence $y$ has $|y_j|<1/2$ and $P(\Shift)y\in\ZZ^{\ZZ}$, it is the canonical observation sequence of a unique point of $\TT^4$.
\end{lemma}
\begin{proof}
Let $e_0$ be the first standard basis vector and $v_j=(\mathsf T^j)^{\mathsf T}e_0$. The vectors $v_0,v_1,v_2,v_3$ form the standard basis and satisfy the recurrence for $P$.
Division by the monic polynomial $P$ and its irreducibility give,
for every integer $m\ge 0$ and every polynomial
$F(X)=\sum_{j=0}^{m}c_jX^j\in\mathbb{Z}[X]$,
\[
 \sum_{j=0}^{m}c_jv_j=0
 \quad\Longleftrightarrow\quad
 P\mid F
 \quad\Longleftrightarrow\quad
 F(\beta)=0.
\]
Expand $\cos^2(\pi t)=(e^{2\pi it}+2+e^{-2\pi it})/4$. Haar integration retains exactly the zero-frequency terms, which are the collision words with their weights $1/4,1/2,1/4$. This proves \eqref{eq:torus-formula}.

For the last statement take $(y_0,y_1,y_2,y_3)\pmod{\ZZ^4}$. Integral forcing makes \eqref{eq:recurrence} homogeneous modulo integers. Forward and backward iteration recover all coordinates modulo integers; the canonical range determines their representatives uniquely.
\end{proof}
For $|t|<1/2$, put
\begin{equation*}
 V(t)=-\log\cos^2(\pi t),\quad V'(t)=2\pi\tan(\pi t),\quad
 V''(t)=2\pi^2\sec^2(\pi t).
\end{equation*}
Thus $V$ is nonnegative, even, and increasing with $|t|$. For $0<\rho\le1/2$,
\begin{equation}\label{eq:potential-away-wall}
 |V'(t)|\le 2\pi\cot(\pi\rho),\qquad
 V''(t)\le 2\pi^2\csc^2(\pi\rho)
 \quad \bigl(|t|\le 1/2-\rho\bigr).
\end{equation}
For $|t|\le h<1/2$,
\begin{equation}
 V(t)\le C_ht^2,\qquad |V'(t)|\le2C_h|t|,\qquad C_h=\pi^2\sec^2(\pi h).
 \label{eq:potential-tail}
\end{equation}
Equation~\eqref{eq:torus-formula} is equivalently $Q_n=\int_{\TT^4}\exp(-\sum_{j=0}^{n-1}V(y_j(x)))\,dx$.

The configurations used below have a finite interval on which the integer forcing is supported, geometrically decaying real tails, and one outgoing oscillatory tail coming from the unit-circle conjugates.  We call such a configuration a packet.  The localized sequence removes only the oscillatory tail; its action measures the analytic cost that will later be compared with the number of admissible placements.

\begin{definition}\label{def:packet}
A packet with parameters $0<\eta<1/4$ and $h>0$ consists of bilateral real sequences $Y,H$, an integer $\ell\ge1$, and $\gamma\in\CC\setminus\{0\}$ such that
\begin{align*}
 &\kappa=P(\Shift)Y\in\ZZ^{\ZZ}\setminus\{0\},\qquad\supp\kappa\subset[0,\ell-1],\\
 &Y_j=H_j+\mathbf 1_{\{j\ge \ell\}}2\Re(\gamma\zeta^{j-\ell}),\qquad |H_j|\le\frac12-2\eta,\\
 &|H_j|\le hr^{\dist(j,[0,\ell-1])}\quad(j\notin[0,\ell-1]).
\end{align*}
Intervals in support statements are intersected with $\ZZ$. The interval $[0,\ell-1]$ is the core. Define
\begin{equation*}
 \Delta=2|\gamma|,\qquad E=\sum_{j\in\ZZ}V(H_j),\qquad
 \mathcal V=\sum_{j\in\ZZ}|V'(H_j)|.
\end{equation*}
Both sums are finite by \eqref{eq:potential-tail} and the geometric tails.
\end{definition}

\begin{lemma}\label{lem:phase-steering}
If $|z|\le8$, every block of eleven consecutive integers contains an integer $m$ such that $|\zeta^m z+1|\le8$.
\end{lemma}
\begin{proof}
The case $z=0$ is immediate. Otherwise write $z=t\theta$, with $0<t\le8$, $|\theta|=1$, and put $f_m =\Re(\theta\zeta^m )$. 
Let $m_0\in\mathbb{Z}$ be arbitrary. Geometric sums and
\eqref{eq:unit-circle-bounds} give
\[
 \left|\frac1{11}\sum_{m=m_0}^{m_0+10} f_m \right|
 <\frac4{33}<\frac18,
 \qquad
 \frac1{11}\sum_{m=m_0}^{m_0+10} f_m ^2
 >\frac12-\frac2{33}>\frac38.
\]
If every $f_m \ge-1/8$, then $(f_m +1/8)(f_m -1)\le0$, so $f_m ^2\le(7/8)f_m +1/8$. Averaging would give an upper bound $15/64<3/8$, a contradiction. Thus some $f_m <-1/8$, and for it
\[
 |\zeta^m z+1|^2=t^2+1+2tf_m \le t^2+1-t/4\le63<64.
\]
The last quadratic is convex on $[0,8]$ and has endpoint values $1$ and $63$.
\end{proof}

For a packet and an integer $N>2\ell+4$, set
\begin{align}
 b_N&=\lfloor N/11\rfloor,\qquad d_N=\lfloor(N-\ell)/2\rfloor,\label{eq:placement-blocks}\\
 e_N&=\frac{2hr^{d_N}}{1-r^N},\qquad u_N=8\Delta+e_N,\label{eq:perturbation-budget}\\
 C_\eta&=2\pi^2\csc^2(\pi\eta),\qquad C_N=E+u_N\mathcal V+C_\eta Nu_N^2,\label{eq:amplifier-cost}\\
 \Theta_N&=\log b_N-C_N.\label{eq:amplifier-gain}
\end{align}
The last expression is used only when $b_N\ge1$.

\begin{proposition}\label{prop:packet-amplification}
Suppose $b_N\ge1$ and $u_N<\eta$. There is $c(P,\eta)>0$ such that, for every integer $M\ge2$, some $5\le n<2NM$ satisfies
\begin{equation}
 \begin{aligned}
 \log(\beta^nQ_n)>{}&M\Theta_N-\log b_N-5\log M-5\log N\\
 &-4\log2+\log c(P,\eta).
 \end{aligned}\label{eq:packet-amplification}
\end{equation}
The constant does not depend on $M,N$, or the packet. If $\Theta_N>0$, then $D_2(\nu_r)<1$. If $\Theta_N>1$, then
\begin{equation}
 I_{1-1/(8N)}(\nu_r)=\infty,\qquad D_2(\nu_r)\le1-\frac1{8N}.
 \label{eq:amplifier-energy}
\end{equation}
\end{proposition}

\begin{proof}
Choose
$
0=\tau_0<\tau_1<\cdots<\tau_{M-1},\  
\tau_j-\tau_{j-1}\in[N,2N-1].
$
For the normalized neutral state
$
\psi_0=1,\ 
\psi_j=\zeta^{\tau_j-\tau_{j-1}}\psi_{j-1}+1,
$
partition the allowed increments into $b_N$ disjoint blocks of eleven
integers and a remainder. By Lemma~\ref{lem:phase-steering}, at each step
there are at least $b_N$ choices preserving $|\psi_j|\le8$; hence there
are at least $b_N^{M-1}$ admissible lists.

For such a list set
$
y_i=\sum_{j=0}^{M-1}Y_{i-\tau_j}.
$
At any site the emitted waves come from a prefix of the list, and if its
last index is $j$, their sum is
$
2\Re(\gamma\zeta^{i-\tau_j-\ell}\psi_j),
$
so its absolute value is at most $8\Delta$; before the first emission it
is zero.

Assign each integer site to a nearest translated core. Its distance from
every other core is at least $d_N$, and the distances to further cores on
either side increase by at least $N$. Thus the nonassigned localized tails
contribute at most
$
2h\sum_{m\ge0}r^{d_N+mN}=e_N.
$
Hence, at a site assigned to the $j$th core,
\begin{equation}
 |y_i-H_{i-\tau_j}|\le u_N,\qquad
 |y_i|<\frac12-\eta.
 \label{eq:localized-perturbation}
\end{equation}
Since the forcing is integral, Lemma~\ref{lem:torus-formula} realizes the
sequence as a canonical torus orbit. Put
$
n=\tau_{M-1}+\ell+4,
$
so $5\le n<2NM$. Taylor's theorem and \eqref{eq:localized-perturbation}
give
\[
V(y_i)\le V(H_{i-\tau_j})
 +u_N|V'(H_{i-\tau_j})|
 +\frac{C_\eta}{2}u_N^2.
\]
The sites assigned to a fixed core form a subset of $\ZZ$, so the
corresponding base action and derivative sums are bounded by
$E$ and $\mathcal V$. Therefore
\begin{equation}
 \sum_{i=0}^{n-1}V(y_i)
 \le M(E+u_N\mathcal V)+\frac{C_\eta n}{2}u_N^2
 \le MC_N.
 \label{eq:placed-action}
\end{equation}
The forcing word determines the placement data. Indeed, writing
$F=\supp\kappa\subset[0,\ell-1]$, gaps inside a translated copy of $F$
are at most $\ell-1$, whereas gaps between distinct copies are at least
$
N-(\ell-1)>\ell.
$
Thus cutting at gaps larger than $\ell$ recovers the clusters, and
subtracting $\min F$ from each first position recovers all starts.
The supports are disjoint, so no cancellation occurs, and the observed
forcing on $0,\ldots,n-5$ contains every copy. Since at most
$(M-1)N+1\le MN$ lengths occur, for some length there are at least
\begin{equation}
 \mathcal N\ge\frac{b_N^{M-1}}{MN}
 \label{eq:decoded-count}
\end{equation}
distinct observed forcing words.

Write $\zeta=e^{i\varphi}$ and let
\begin{equation}
 R_P=\begin{pmatrix}
 1&1&1&0\\
 \beta&r&\cos\varphi&\sin\varphi\\
 \beta^2&r^2&\cos(2\varphi)&\sin(2\varphi)\\
 \beta^3&r^3&\cos(3\varphi)&\sin(3\varphi)
 \end{pmatrix}.
 \label{eq:mode-matrix}
\end{equation}
By Lemma~\ref{lem:homogeneous-modes}, $R_P$ is invertible. Fix
$0<\rho_{\rm box}\le\eta/64$. Around each of the $\mathcal N$ orbit
centers add
$
z_i=a_\beta\beta^i+a_rr^i+a_c\cos(i\varphi)+a_s\sin(i\varphi),
$
with
\[
|a_\beta|\le\frac{\rho_{\rm box}\beta^{-(n-1)}}{n},
\qquad
|a_r|,|a_c|,|a_s|\le\frac{\rho_{\rm box}}{n}.
\]
Then $|z_i|\le4\rho_{\rm box}/n$ for $0\le i<n$, so the perturbed
observations remain canonical and at distance at least $\eta/2$ from the
centering boundary. Since
\[
|V'(t)|\le2\pi\cot(\pi\eta/2)\le\frac4\eta
\qquad
\left(|t|\le\frac12-\frac\eta2\right),
\]
the action increases by at most
$
\frac{16\rho_{\rm box}}{\eta}<1.
$
The corresponding initial-coordinate box has volume
$
|\det R_P|(2\rho_{\rm box})^4\beta^{-(n-1)}n^{-4}.
$
Its projection to $\TT^4$ is injective because coordinate differences are
bounded by $8\rho_{\rm box}/n<1$. Moreover, homogeneous perturbations do
not change the forcing word while the observations remain canonical;
hence boxes corresponding to distinct observed words are disjoint,
including across fundamental-cube boundaries. Integrating
\eqref{eq:torus-formula} over these boxes and using
\eqref{eq:placed-action} yields
\begin{equation}
 \beta^nQ_n
 \ge c(P,\eta)n^{-4}\mathcal N e^{-MC_N},
 \qquad
 c(P,\eta)
 =\beta|\det R_P|(2\rho_{\rm box})^4e^{-1}>0.
 \label{eq:torus-box-lower}
\end{equation}
Combining \eqref{eq:decoded-count} with $n<2NM$ gives
\eqref{eq:packet-amplification}; in particular $c(P,\eta)$ is independent
of $M,N$, and the packet.

If $\Theta_N>1$, let $s=1-1/(8N)$. Since $\log\beta<1$ and $n<2NM$,
\[
(1-s)n\log\beta<\frac M4.
\]
Thus, for all sufficiently large $M$, \eqref{eq:packet-amplification}
implies $\log(\beta^{sn}Q_n)>0$. Corollary~\ref{cor:finite-collision}
therefore gives
$
I_{1-1/(8N)}(\nu_r)=\infty.
$
Thus, 
$
D_2(\nu_r)\le1-\frac1{8N}.
$
Finally, if merely $\Theta_N>0$, choose
$
0<a<\min\left\{1,\frac{\Theta_N}{4N\log\beta}\right\},
$ and $s=1-a.
$
Then
$
(1-s)n\log\beta<\frac{M\Theta_N}{2},
$
so after subtracting this penalty the right-hand side of
\eqref{eq:packet-amplification} is positive for all sufficiently large
$M$. Another application of Corollary~\ref{cor:finite-collision} yields
$D_2(\nu_r)<1$.
\end{proof}

\begin{corollary}\label{cor:packet-criterion}
Fix either quartic parameter and constants $0<\eta<1/4$ and $h>0$. Suppose there is a sequence of packets in the sense of Definition~\ref{def:packet}, with these same constants, core lengths $\ell_k$, actions $E_k$, and neutral amplitudes $\Delta_k$, such that
\[
 \Delta_k\longrightarrow0,
 \qquad \Delta_k\ell_k\longrightarrow0,
 \qquad E_k=o\bigl(\log(1/\Delta_k)\bigr).
\]
Then $D_2(\nu_{\beta^{-1}})<1$.
\end{corollary}
\begin{proof}
The common boundary margin and the geometric-tail estimate imply
\[
 \mathcal V_k=\sum_{j \in \mathbb{Z}}|V'(H_j^{(k)})|\le C(\ell_k+1)
\]
with $C$ independent of $k$. Indeed, \eqref{eq:potential-away-wall} bounds the derivative on the core. On each tail, there are only a fixed number of sites before $hr^j\le1/4$, and \eqref{eq:potential-tail} bounds the remaining derivative sum by a geometric series.

Set $N_k=\lceil\Delta_k^{-1}\rceil$. Then $\ell_k/N_k\to0$, so $N_k>2\ell_k+4$ eventually. With the notation \eqref{eq:placement-blocks}--\eqref{eq:amplifier-gain}, one has $d_{N_k}\ge N_k/4-1$ for all sufficiently large $k$, and
\[
 \frac{e_{N_k}}{\Delta_k}
 \le\frac{2hN_kr^{N_k/4-1}}{1-r^{N_k}}
 \longrightarrow0.
\]
Consequently $u_{N_k}=O(\Delta_k)$ and
\[
 u_{N_k}\mathcal V_k\longrightarrow0,
 \qquad N_ku_{N_k}^2\longrightarrow0.
\]
Since $b_{N_k}\ge N_k/22$ eventually,
\[
 \Theta_{N_k}
 \ge\log(1/\Delta_k)-E_k-\log22-o(1)
 \longrightarrow+\infty.
\]
In particular, $u_{N_k}<\eta$ and $\Theta_{N_k}>0$ for some $k$. Proposition~\ref{prop:packet-amplification} gives the conclusion.
\end{proof}

\section{Reciprocal packets and an algebraic seed}\label{sec:reciprocal}
For every root $\lambda$ of $P$, put $C_\lambda=P'(\lambda)^{-1}$. Reciprocity gives
\begin{equation}
 P'(\lambda^{-1})=-\lambda^{-2}P'(\lambda),\qquad C_r=-\beta^2C_\beta.
 \label{eq:reciprocity-residues}
\end{equation}
Let $K\in\ZZ[X,X^{-1}]$ be nonzero and satisfy
\begin{equation}
 K(X^{-1})=\sigma X^{-2c_K}K(X),\qquad\sigma\in\{-1,1\},\quad c_K\in\ZZ,
 \label{eq:reciprocal-polynomial}
\end{equation}
with $K(\zeta)\ne0$. Define
\begin{equation}
 \begin{gathered}
 j_*=4-2c_K,\qquad \xi_i=-C_\beta\beta^{i-1}K(\beta),\\
 \chi_i=2\Re\bigl(C_\zeta\zeta^{i-1}K(\zeta)\bigr),\qquad
 \Delta=2|C_\zeta K(\zeta)|.
 \end{gathered}\label{eq:trace-coordinates}
\end{equation}
Irreducibility implies $K(\beta)\ne0$: otherwise a monomial multiple of $K$ would be divisible by $P$ and vanish at $\zeta$.

\begin{lemma}\label{lem:trace-identity}
For every $i\in\ZZ$,
\[
 \xi_i\equiv\sigma \xi_{j_*-i}+\chi_i\pmod{\ZZ},\qquad \chi_{j_*-i}=-\sigma \chi_i.
\]
\end{lemma}
\begin{proof}
Let $\mathcal T_m=\sum_{P(\lambda)=0}\lambda^m/P'(\lambda)$. Lagrange interpolation of $1,X,X^2,X^3$, followed by comparison of the coefficient of $X^3$, gives $\mathcal T_0=\mathcal T_1=\mathcal T_2=0$ and $\mathcal T_3=1$. The monic recurrence for $P$, with constant coefficient one, propagates integrality both forward and backward. Thus $\mathcal T_m\in\ZZ$ for all $m\in\ZZ$, and
\[
 \sum_{P(\lambda)=0}C_\lambda\lambda^{i-1}K(\lambda)\in\ZZ.
\]
The expanding term is $-\xi_i$. By \eqref{eq:reciprocal-polynomial} and \eqref{eq:reciprocity-residues}, the stable term is
\[
 C_rr^{i-1}K(r)=-\sigma C_\beta\beta^{3-i-2c_K}K(\beta)=\sigma \xi_{j_*-i}.
\]
The remaining two terms sum to $\chi_i$, proving the congruence. If $z_0=C_\zeta K(\zeta)$, the same reciprocal identities give $\chi_i=z_0\zeta^{i-1}-\sigma z_0\zeta^{j_*-i-1}$. Replacing $i$ by $j_*-i$ proves the second assertion.
\end{proof}

\begin{definition}\label{def:reciprocal-packet}
A centered reciprocal packet associated with $K$ is a sequence $Y_i=\can(\xi_i)$ with no half-integer values before centering, whose nonzero integer forcing is supported in an interval $[a,b]$ with $a+b=j_*$, and for which
\[
Y_i=\xi_i\quad(i<a),\qquad Y_i=\sigma \xi_{j_*-i}+\chi_i\quad(i>b).
\]
Its localization is
\[
 H_i=\begin{cases}Y_i,&i\le b,\\\sigma \xi_{j_*-i},&i>b.\end{cases}
\]
Write
\[
 \ell=b-a+1\ge4,\quad h_0=|\xi_a|,\quad
 E=\sum_{i\in\ZZ}V(H_i),\quad \mathcal V_{\mathrm{core}}=\sum_{i=a}^b|V'(Y_i)|.
 \]
The two real tails have absolute values $h_0r^j$, $j\ge1$.
\end{definition}
A translation of the core to $[0,\ell-1]$ gives Definition~\ref{def:packet} once a common positive localized boundary distance has been specified. Translation changes the phase but not the amplitude of the outgoing wave.

\begin{lemma}\label{lem:mirror}
Suppose the core of a centered reciprocal packet has boundary distance at least $\vartheta>0$. Let $L>\ell$ be even, and set
\[
\ell_{\mathrm{gap}}=L-\ell,\quad\Delta^+=\Delta|1-\zeta^L|,\quad\rho_L=h_0r^{\ell_{\mathrm{gap}}+1},\quad U=\Delta+2h_0r.
\]
Assume
\[
\Delta^++\rho_L\le\vartheta/2,\qquad U<1/2,\qquad\Delta^++h_0r<1/2.
\]
Replacing $K$ by $K^+=(1-X^{-L})K$, centering its expanding mode and localizing after the core $[a,b+L]$ produces another centered reciprocal packet. Its forcing is nonzero, integral and supported in that core. Its neutral amplitude is $\Delta^+$, and its real-tail coefficient is $h_0^+=(1-r^L)h_0\le h_0$. Its global localized boundary distance is at least
\begin{equation}
 \min\{\vartheta-\Delta^+-\rho_L,\;1/2-U,\;1/2-h_0r\}.\label{eq:mirror-margin}
\end{equation}
With
\[
 C_{\mathrm{core}}=2\pi^2\csc^2(\pi\vartheta/2),\qquad C_U=2\pi^2\sec^2(\pi U),
\]
the full localized action satisfies
\begin{equation}
 \begin{aligned}
 E^+\le{}&2E+\mathcal V_{\mathrm{core}}(\Delta^++2\rho_L)+C_{\mathrm{core}}\ell(\Delta^++\rho_L)^2\\
 &+C_U\left\{h_0^2\ell_{\mathrm{gap}}r^{\ell_{\mathrm{gap}}+1}+\frac{2\Delta h_0r}{1-r}
 +\Delta^2\left(\frac{\ell_{\mathrm{gap}}}4+\frac{1}{2|1-\zeta^2|}\right)\right\}.
 \end{aligned}\label{eq:mirror-action}
\end{equation}
\end{lemma}
\begin{proof}
The new reciprocal sign is $-\sigma$ and its reflection index is $j_*+L$, since $K^+(X^{-1})=-\sigma X^{L-2c_K}K^+(X)$. Evenness of $L$ ensures that the new center parameter $c_K-L/2$ is integral. Directly from \eqref{eq:trace-coordinates},
\[
 \xi_i^+=\xi_i-\xi_{i-L}=(1-r^L)\xi_i,\qquad \chi_i^+=\chi_i-\chi_{i-L}.
\]
The neutral amplitude is exactly $\Delta^+$ and is nonzero because $\zeta$ is not a root of unity.

First consider the two copies of the old core. For $a\le i\le b$, the number $Y_i-\xi_{i-L}$ is congruent to $\xi_i^+$. Since $|\xi_{i-L}|\le\rho_L$, it remains canonical and equals $Y_i^+$. On the second core write
\[
 i=L+j_*-k,\qquad a\le k\le b.
\]
By Lemma~\ref{lem:trace-identity}, the trace identities give
\[
 Y_i^+=\sigma\bigl(-Y_k+\xi_{k-L}+\chi_k^+\bigr).
 \]
The right side is congruent to $\xi_i^+$ and differs from $-\sigma Y_k$ by at most $\rho_L+\Delta^+\le\vartheta/2$, so is the canonical representative. Taylor's theorem and evenness of $V$ bound the action of the two copies by
\begin{equation}
 2\sum_{i=a}^bV(Y_i)+\mathcal V_{\mathrm{core}}(\Delta^++2\rho_L)
 +C_{\mathrm{core}}\ell(\Delta^++\rho_L)^2.
 \label{eq:core-action}
\end{equation}
Indeed, the quadratic remainder is at most $(C_{\mathrm{core}}\ell/2)(\rho_L^2+(\rho_L+\Delta^+)^2)$.

On the gap, for $i=b+j$, $1\le j\le\ell_{\mathrm{gap}}$, put $x_j=\sigma \xi_{j_*-i}$ and $z_j=-\xi_{i-L}$. Then $|x_j|=h_0r^j$, $|z_j|=h_0r^{\ell_{\mathrm{gap}}+1-j}$, and
\[
Y_i^+=x_j+z_j+\chi_i.
\]
The trace congruence and the bound $|x_j|+|z_j|+|\chi_i|\le U<1/2$ justify this canonical expression. On $[-U,U]$, $V''\le C_U$ and $|V'(w)|\le C_U|w|$. Whenever $|x|+|z|+|w|\le U$, integration of the mixed derivative and then Taylor expansion in $w$ give
\begin{equation}
 V(x+z+w)-V(x)-V(z)\le C_U\left(|xz|+(|x|+|z|)|w|+\frac{w^2}{2}\right).
 \label{eq:gap-taylor}
\end{equation}
For the mixed difference without the oscillatory term, use $\int_0^x\int_0^zV''(u+v)\,dv\,du$; all integration paths stay in $[-U,U]$. The first two error terms in \eqref{eq:gap-taylor} satisfy
\[
 \sum_{j=1}^{\ell_{\mathrm{gap}}}|x_jz_j|=h_0^2\ell_{\mathrm{gap}}r^{\ell_{\mathrm{gap}}+1},\qquad
 \sum_{j=1}^{\ell_{\mathrm{gap}}}(|x_j|+|z_j|)|\chi_{b+j}|\le\frac{2\Delta h_0r}{1-r}.
\]
Writing $\chi_{b+j}=\Delta\cos(\theta+j\varphi)$ and using a geometric sum at $\zeta^2$ gives
\begin{equation}
 \frac12\sum_{j=1}^{\ell_{\mathrm{gap}}}\chi_{b+j}^2\le\Delta^2\left(\frac{\ell_{\mathrm{gap}}}4+\frac{1}{2|1-\zeta^2|}\right).
 \label{eq:gap-square-sum}
\end{equation}

Outside the two cores, for $i<a$, $(1-r^L)\xi_i$ is canonical. For $i>b+L$, the new trace congruence gives
\[
 Y_i^+=-\sigma \xi^+_{j_*+L-i}+\chi_i^+,
\]
which is canonical by $h_0r+\Delta^+<1/2$. The localized exterior tails have absolute values $h_0^+r^j$, $j\ge1$.

Let $T_0=\sum_{j=1}^\infty V(h_0r^j)$. The gap base terms $V(x_j)+V(z_j)$ sum to at most $2T_0$, and the two new exterior tails sum to at most $2T_0$. These are bounded by the four old tail sums already contained in $2E$. Combining this with \eqref{eq:core-action}--\eqref{eq:gap-square-sum} proves \eqref{eq:mirror-action}. The same regional estimates give \eqref{eq:mirror-margin}.

It remains to check the support and nonvanishing of the forcing. Centering a homogeneous sequence changes it by an integer sequence, so the new forcing is integral. For all $i\le a+3$ one has $Y_i^+=Y_i-\xi_{i-L}$, using the left-tail expression and the first four positions of the first core; these exist because $\ell\ge4$. For $j<a$, every coordinate in the five-term recurrence has index at most $a+3$. The new forcing there equals the old forcing and vanishes. For $j>b+L$, all coordinates lie in the homogeneous stable--neutral tail.

If the forcing vanished everywhere, Lemma~\ref{lem:homogeneous-modes} would make $Y^+$ zero, since it decays at $-\infty$ and is bounded at $+\infty$. This contradicts its nonzero expanding tail: $K^+(\beta)=(1-r^L)K(\beta)\ne0$.
\end{proof}

Define the integral Laurent polynomials
\[
U_1(X)=1+X+X^{-1},\qquad U_2(X)=X+X^{-1},\qquad K_i(X)=U_i(X)^{10}.
\]
Their values are
\[
 \begin{array}{c|cc}
 &U_i(\beta_i)&U_i(\zeta_i)\\\hline
 i=1&(3+\sqrt{13})/2&(3-\sqrt{13})/2\\
 i=2&1+\sqrt2&1-\sqrt2.
 \end{array}
\]
The product in each row is $-1$. In particular, $U_i(\beta_i)>1$ and $0<|U_i(\zeta_i)|<1/2$. These seeds are chosen only for the qualitative existence argument; the quantitative constructions in Appendix~\ref{app:upper} use separate parameter-specific seeds optimized for the explicit bounds.

\begin{lemma}\label{lem:algebraic-seed}
For either polynomial, $K_i$ gives a centered reciprocal packet with finite core, nonzero integer forcing, finite localized action, and a positive global localized boundary distance $\vartheta_{\mathrm{seed}}$. Its real-tail coefficient $h$ and neutral amplitude satisfy
\[
 0<h<1/100,\qquad 0<\Delta_{\mathrm{seed}}<1/100.
 \]
The positive boundary distance is obtained without evaluating any core coordinate or the minimum itself.
\end{lemma}
\begin{proof}
Suppress the parameter index. The polynomial $K=U^{10}$ has $\sigma=1$, $c_K=0$ and $j_*=4$ in \eqref{eq:reciprocal-polynomial}. By \eqref{eq:unit-circle-bounds},
\[
 0<\Delta_{\mathrm{seed}}=2|C_\zeta K(\zeta)|<2\cdot2^{-10}=1/512<1/100.
\]
For every $j\in\ZZ$, the embedding $\QQ(\beta)\to\CC$ sending $\beta$ to $\zeta$ sends $\xi_j=-C_\beta\beta^{j-1}K(\beta)$ to $-C_\zeta\zeta^{j-1}K(\zeta)$, whose modulus is $\Delta_{\mathrm{seed}}/2<1/200$. If $\xi_j\in\ZZ+1/2$, it would be a rational number fixed by the embedding, with modulus at least $1/2$, a contradiction. Thus $Y_j=\can(\xi_j)$ never has a half-integer value before centering.

Choose $a\le0$ sufficiently negative that $h=|\xi_a|<1/100$, and put $b=4-a$. For $j\le a+3$, $|\xi_j|\le\beta^3h<8/100<1/2$, hence $Y_j=\xi_j$. For $j>b$, the trace identity gives
\[
 Y_j=\xi_{4-j}+\chi_j,\qquad |\xi_{4-j}|+|\chi_j|\le hr^{j-b}+\Delta_{\mathrm{seed}}<1/2.
\]
The forcing is integral and vanishes for $j<a$ and $j>b$, by the five-coordinate recurrence and these homogeneous tail formulas. It is nonzero by Lemma~\ref{lem:homogeneous-modes}, since $K(\beta)\ne0$. The core length $\ell_{\mathrm{seed}}=b-a+1\ge5$ is finite.

Localize after $b$. The tails have absolute values $hr^j$, $j\ge1$, and
\[
\vartheta_{\mathrm{seed}}=\min\left\{\min_{a\le j\le b}\left(\frac12-|Y_j|\right),\frac12-h\right\}>0.
\]
This is the minimum of a finite nonempty set of positive numbers, and its positivity has already been proved without evaluating it. The core action is finite, and the tails converge by \eqref{eq:potential-tail}.
\end{proof}

\section{Controlled returns and the proof of the main theorem}\label{sec:returns}
Put $\omega=\zeta^2$.  For every integer $R\ge1$, let $q(R)$ be the least positive integer attaining
\[
 \min_{1\le q\le R}\operatorname{err}(q),
 \qquad
 \operatorname{err}(q)=|1-\omega^q|.
\]
The minimizer is taken over a finite set; no numerical value of it will be used.

\begin{lemma}\label{lem:best-returns}
For every integer $R\ge1$,
\[ 0<\operatorname{err}(q(R))\le\frac{2\pi}{R},
 \qquad
 q(R)\longrightarrow\infty.
\]
If $R_*\ge1$, $q_*=q(R_*)$, and
$\operatorname{err}(q(R))<\operatorname{err}(q_*)$, then
$q(R)>R_*\ge q_*$.
\end{lemma}
\begin{proof}
Write $\omega=e^{2\pi i\alpha}$ with $\alpha$ irrational.  Among the $R+1$ fractional parts of
$0,\alpha,\ldots,R\alpha$, two lie in the same one of $R$ equal half-open intervals.  Their difference gives $1\le q\le R$ with
$\|q\alpha\|_{\RR/\ZZ}\le1/R$, and hence
$|1-e^{2\pi iq\alpha}|\le2\pi/R$.  No error vanishes because $\omega$ is not a root of unity.  Since every fixed finite set of denominators has a positive minimum error whereas $2\pi/R\to0$, the minimizers tend to infinity.  For the final assertion, every $1\le q\le R_*$ has error at least $\operatorname{err}(q_*)$.  A strict improvement must therefore use a denominator larger than $R_*$, and hence larger than $q_*$.
\end{proof}

Take the packet in Lemma~\ref{lem:algebraic-seed} and fix
\[
 \eta=\min\{\vartheta_{\mathrm{seed}}/8,1/100\}>0.
\]
Keep $h$ as a common upper bound for the real-tail coefficients.  Choose $R_0$ sufficiently large, put $q_0=q(R_0)$ and $L_0=2q_0$, and apply Lemma~\ref{lem:mirror} once.  The choice may be made so that the resulting packet has core length $\ell_0$ and amplitude $\Delta_0$ satisfying
\begin{align}
 \ell_0&=\ell_{\mathrm{seed}}+2q_0<3q_0,\label{eq:preparation-length}\\
 0<\Delta_0&=\Delta_{\mathrm{seed}}\operatorname{err}(q_0)
 <\min\{e^{-16},\eta/4\},\\
 \frac{hr^{\ell_0}}{1-r^{\ell_0}}&<\eta/4,
 \qquad
 \text{localized boundary distance at least }4\eta.\label{eq:preparation-margin}
\end{align}
Indeed, $q_0\to\infty$ and $\operatorname{err}(q_0)\to0$.  We may first require $2q_0>\ell_{\mathrm{seed}}$.  The perturbation of the old core is at most
$\Delta_{\mathrm{seed}}\operatorname{err}(q_0)+hr^{2q_0-\ell_{\mathrm{seed}}+1}$, which tends to zero.  The gap parameter in the mirror lemma is bounded by
$\Delta_{\mathrm{seed}}+2hr<3/100$, while the exterior real tail is bounded by $hr<1/100$; both are therefore uniformly separated from the boundary of $[-1/2,1/2)$.  Since $\vartheta_{\mathrm{seed}}/2\ge4\eta$, all the requirements in \eqref{eq:preparation-length}--\eqref{eq:preparation-margin} hold for one sufficiently large $R_0$.  Fix such a choice.

\begin{theorem}\label{thm:controlled-iteration}
Under the conditions
\eqref{eq:preparation-length}--\eqref{eq:preparation-margin},
there are centered reciprocal packets indexed by $k\ge0$ with core
lengths $\ell_k$, neutral amplitudes $\Delta_k$, real-tail coefficients
at most $h$, and a common localized boundary distance at least $3\eta$.
There are constants $C_E,C_{\mathcal V}<\infty$, independent of $k$, such that
\begin{align}
 0<\Delta_k&\le e^{-16\cdot2^k},
 &\log(1/\Delta_k)&\ge10(5/2)^k,\label{eq:iteration-amplitude}\\
 E_k&\le C_E2^k,
 &\mathcal V_k&\le C_{\mathcal V}(\ell_k+1),\label{eq:iteration-action}\\
 \ell_k&\le\frac32\Delta_{k-1}^{-3/2}\quad(k\ge1),
 &\Delta_k\ell_k&\longrightarrow0.\nonumber
\end{align}
Here $\mathcal V_k=\sum_{j\in\mathbb{Z}}
\bigl|V'(H_j^{(k)})\bigr|$ is the full derivative sum of the localized packet.
\end{theorem}
\begin{proof}
For $k\ge1$ set
\begin{equation}\label{eq:return-choice}
 m_k=3^k,
 \qquad
 R_k=\left\lfloor\frac{\Delta_{k-1}^{-3/2}}{2m_k}\right\rfloor,
 \qquad
 q_k=q(R_k),
 \qquad
 L_k=2m_kq_k.
\end{equation}
The $k$th operation will be the reciprocal mirror with delay $L_k$.  For the comparison with the preparatory step put $m_0=1$, $\Delta_{-1}=\Delta_{\mathrm{seed}}$, and regard $q_0$ as fixed above.

For the return estimates, assume inductively that $\Delta_{k-1}\le e^{-16\cdot2^{k-1}}$.  The crude fixed bounds $\log28<4$ and $2\log3<3$ give
\begin{equation}\label{eq:return-fixed-comparison}
 28\cdot9^k\sqrt{\Delta_{k-1}}<1.
\end{equation}
Indeed, the logarithm of the left side is at most
$4+3k-8\cdot2^{k-1}<0$.  The quantity inside the floor in \eqref{eq:return-choice} is therefore larger than two, and Lemma~\ref{lem:best-returns} gives
\[
 \operatorname{err}(q_k)
 \le\frac{2\pi}{R_k}
 \le8\pi m_k\Delta_{k-1}^{3/2}
 <28m_k\Delta_{k-1}^{3/2}.
\]
At the preceding step,
\[
 \Delta_{k-1}
 =\Delta_{k-2}|1-\omega^{m_{k-1}q_{k-1}}|
 \le m_{k-1}\Delta_{k-2}\operatorname{err}(q_{k-1}),
\]
with the same identity interpreted using the preparatory packet when $k=1$.  Since $\Delta_{k-2}<1$,
$\operatorname{err}(q_{k-1})>\Delta_{k-1}/m_{k-1}$.  On the other hand, \eqref{eq:return-fixed-comparison} and $m_km_{k-1}\le9^k$ imply
\[
 \operatorname{err}(q_k)
 <\frac{\Delta_{k-1}}{m_{k-1}}
 <\operatorname{err}(q_{k-1}).
\]
Hence $q_k>q_{k-1}$ by Lemma~\ref{lem:best-returns}.

The induction invariant $\ell_{k-1}<3^kq_{k-1}$ holds for $k=1$ by \eqref{eq:preparation-length}.  Since the best-return denominators increase,
\begin{equation}\label{eq:delay-separation}
 L_k>2\ell_{k-1},
 \qquad
 L_k\le\Delta_{k-1}^{-3/2},
 \qquad
 \ell_k=\ell_{k-1}+L_k<3^{k+1}q_k.
\end{equation}
Thus the delay separates the two copies of the old core.  Moreover,
\begin{equation}\label{eq:amplitude-recurrence}
 \begin{aligned}
 0<\Delta_k
 &=\Delta_{k-1}|1-\omega^{m_kq_k}|\\
 &\le m_k\Delta_{k-1}\operatorname{err}(q_k)
 <28\cdot9^k\Delta_{k-1}^{5/2}
 \le\Delta_{k-1}^2.
 \end{aligned}
\end{equation}
The last inequality is \eqref{eq:return-fixed-comparison}; positivity follows because $\omega$ is not a root of unity.  Consequently
$\Delta_k\le\Delta_0^{2^k}\le e^{-16\cdot2^k}$.

These estimates also keep the iterates uniformly away from the boundary. Let
\[
 g_k=L_k-\ell_{k-1}>\ell_{k-1},
 \qquad
 \rho_k=hr^{g_k+1}.
\]
Since \eqref{eq:delay-separation} gives $\ell_j>3\ell_{j-1}$,
\[
 \sum_{k=1}^\infty \rho_k
 \le\frac{hr^{\ell_0}}{1-r^{\ell_0}}<\frac\eta4.
\]
Also $\Delta_k\le\Delta_0^{2^k}\le\Delta_0^{2k}$, and hence
\[
 \sum_{k=1}^\infty \Delta_k
 \le\frac{\Delta_0^2}{1-\Delta_0^2}
 <\Delta_0<\frac\eta4.
\]
Starting from boundary distance $4\eta$, the cumulative loss on all old cores is therefore less than $\eta/2$.  The gap parameter is at most
$\Delta_{k-1}+2hr<3/100$, while the exterior real tail is at most $hr<1/100$; these regions stay much farther than $4\eta$ from the boundary.  Thus every application of Lemma~\ref{lem:mirror} is legitimate, the localized boundary distance remains larger than $3\eta$, the forcing remains nonzero and integral, and the real-tail coefficient does not increase.

For the amplitude and action estimates, write $x_k=\log(1/\Delta_k)$. Then \eqref{eq:amplitude-recurrence} gives
\[
 x_k\ge\frac52x_{k-1}-\log28-2k\log3.
\]
After division by $(5/2)^k$ and summation,
\[
 \frac{x_k}{(5/2)^k}
 \ge x_0-
 \sum_{j=1}^\infty \frac{\log28+2j\log3}{(5/2)^j}
 >x_0-6\ge10.
\]
Here we used
$\sum_{j=1}^{\infty}(2/5)^j=2/3$
and
$\sum_{j=1}^{\infty}j(2/5)^j=10/9$ together with the same fixed logarithmic bounds as above.  This proves the second estimate in \eqref{eq:iteration-amplitude}.

It remains to show that the action does not acquire a comparable growth rate.  In the $k$th mirror, put
$\ell=\ell_{k-1}$, $\Delta=\Delta_{k-1}$ and $g=g_k$.  From \eqref{eq:delay-separation} and \eqref{eq:amplitude-recurrence},
\[
 \ell\Delta_k<1,
 \qquad
 \Delta^2g\le\sqrt\Delta<1.
\]
Indeed $\ell<L_k/2\le\Delta^{-3/2}/2$, while
$\Delta_k<28\cdot9^k\Delta^{5/2}$ and
\eqref{eq:return-fixed-comparison} applies.  Put
\[
 R_r=\sum_{j=1}^\infty jr^j=\frac{r}{(1-r)^2},\qquad
 B_0=2\pi\cot(\pi\eta),\qquad
 A_0=2\pi^2\csc^2(\pi\eta).
\]
Since $g>\ell$,
\[
 \ell\rho_k\le hR_r,
 \qquad
 \ell\rho_k^2\le h^2R_r,
 \qquad
 gr^{g+1}\le R_r.
\]
The common boundary distance is at least $3\eta$, hence
$\mathcal V_{\mathrm{core}}\le B_0\ell$ and the core curvature
constant in Lemma~\ref{lem:mirror} is at most $A_0$.  Moreover the gap
parameter is at most $3/100$, so its curvature constant is bounded by a
fixed number $C_g$.  The five error terms in \eqref{eq:mirror-action}
are therefore bounded, respectively, by
\[
 B_0(1+2hR_r),\qquad
 2A_0(1+h^2R_r),\qquad
 C_gh^2R_r,\qquad
 C_g\frac{2hr}{1-r},
\]
and
\[
 C_g\left(\frac14+\frac{1}{2|1-\zeta^2|}\right).
\]
Here we used $\ell\Delta_k^2\le\ell\Delta_k<1$ and
$\Delta^2g<1$.  Their sum is a constant $C_*<\infty$ independent of
$k$, and hence
\[
 E_k\le2E_{k-1}+C_*,\qquad E_k\le C_E2^k
\]
for a fixed $C_E$.  Finally, the core contributes at most $B_0\ell_k$
to the full derivative sum.  On either real tail,
\eqref{eq:potential-tail} gives a uniformly summable geometric bound.
Thus $\mathcal V_k\le C_{\mathcal V}(\ell_k+1)$ for a fixed
$C_{\mathcal V}$.

Finally, $L_k>2\ell_{k-1}$ implies
$\ell_k<3L_k/2\le(3/2)\Delta_{k-1}^{-3/2}$.  Combining this with \eqref{eq:amplitude-recurrence} gives
\[
 \Delta_k\ell_k
 \le42\cdot9^k\Delta_{k-1}
 \longrightarrow0,
\]
because $\Delta_{k-1}\le e^{-16\cdot2^{k-1}}$.  This completes the proof.
\end{proof}

\begin{proof}[Proof of Theorem~\ref{thm:main}]
Fix either quartic parameter and translate the packets in Theorem~\ref{thm:controlled-iteration} so that their cores are $[0,\ell_k-1]$.  They satisfy Definition~\ref{def:packet} with the same $\eta$ and $h$, because their localized boundary distance is at least $3\eta>2\eta$.  Put
\[
 N_k=\lceil\Delta_k^{-1}\rceil.
\]
Then
\begin{equation}\label{eq:final-scale-limits}
 \frac{\ell_k}{N_k}\le\Delta_k\ell_k\longrightarrow0,
 \qquad
 N_k\Delta_k^2\le\Delta_k+\Delta_k^2\longrightarrow0.
\end{equation}
In particular $N_k>2\ell_k+4$ for all sufficiently large $k$.  For the geometric part of $u_{N_k}$ in \eqref{eq:perturbation-budget}, we then have
$d_{N_k}\ge N_k/4-1$.  Since $\Delta_k^{-1}\le N_k$,
\[
 0\le
 \frac{e_{N_k}}{\Delta_k}
 \le
 \frac{2hN_kr^{N_k/4-1}}{1-r^{N_k}}
 \longrightarrow0.
\]
Thus $u_{N_k}\le9\Delta_k$ eventually, and \eqref{eq:iteration-action} together with \eqref{eq:final-scale-limits} gives
\[
 u_{N_k}\mathcal V_k\longrightarrow0,
 \qquad
 C_\eta N_ku_{N_k}^2\longrightarrow0.
\]
In particular $u_{N_k}<\eta$ eventually.

For $N_k\ge22$ one has $b_{N_k}\ge N_k/22$.  Therefore
\[
 \begin{aligned}
 \Theta_{N_k}
 &\ge \log(1/\Delta_k)-\log22-C_E2^k-o(1)\\
 &\ge 10(5/2)^k-C_E2^k-\log22-o(1)
 \longrightarrow+\infty.
 \end{aligned}
\]
The growth $(5/2)^k$ therefore defeats the doubling of the action.  For some finite $k$ all the hypotheses of Proposition~\ref{prop:packet-amplification} hold with $\Theta_{N_k}>1$.  Equation~\eqref{eq:amplifier-energy} then gives
$D_2(\nu_r)\le1-1/(8N_k)<1$.  The argument applies to both $P_1$ and $P_2$, proving the theorem.
\end{proof}

The three estimates in Theorem~\ref{thm:controlled-iteration} have distinct roles.  The logarithmic decay of $\Delta_k$ provides the placement gain, the $O(2^k)$ action bound controls its analytic cost, and $\Delta_k\ell_k\to0$ prevents the small neutral error from accumulating over a growing core.  The increasing best-return denominators are also essential: they ensure that each new delay separates the two copies of the previous core.  No bounded-type assumption or quantitative irrationality exponent is imposed on the angle of $\zeta$.  The argument uses one neutral conjugate pair and makes no claim for Salem numbers of higher degree.

\begin{corollary}\label{cor:no-L2}
Neither $\nu_{r_1}$ nor $\nu_{r_2}$ has an $L^2(\RR)$ density.
\end{corollary}
\begin{proof}
If $f\in L^2(\RR)$ is a compactly supported probability density and $X,X'$ are independent with density $f$, then
\[
 \PP(|X-X'|\le t)
 =\int_{-t}^t(f*\widetilde f)(u)\,du
 \le2t\|f\|_2^2,
 \qquad \widetilde f(u)=f(-u).
\]
A dyadic, or equivalently geometric, shell decomposition then gives finite $s$-energy for every $0<s<1$, so $D_2=1$.  This contradicts Theorem~\ref{thm:main}.
\end{proof}

\appendix

\section{Explicit upper bounds}\label{app:upper}

The explicit upper bounds proved in this appendix are obtained by making the seed-and-return construction used in the proof of Theorem~\ref{thm:main} quantitative. We fix the seed and return choices and apply Proposition~\ref{prop:packet-amplification} with explicit parameters. The interval comparisons given below use outward-rounded rational bounds; Appendix~\ref{app:verification} provides the corresponding arithmetic details.

For later use we record an explicit version of the constant in Proposition~\ref{prop:packet-amplification}. As observed in its proof, the torus-box construction is valid for any fixed
\[
0<\rho_{\rm box}\le \eta/64.
\]
For $\eta=3/100$, we take
\[
\rho_{\rm box}=10^{-5}.
\]
If $u=\beta+r$ and $v=\zeta+\bar\zeta$, the real Vandermonde determinant in \eqref{eq:mode-matrix} satisfies
\[
 |\det R_P|=(\beta-r)\sin\varphi\,(u-v)^2,
 \qquad \zeta=e^{i\varphi}.
\]
Here, \eqref{eq:root-bounds} gives $\beta-r>16/15$, \eqref{eq:unit-circle-bounds} gives $\sin\varphi>3/4$, and \eqref{eq:trace-roots} gives $(u-v)^2\ge8$. Thus $|\det R_P|>1$. Hence, by \eqref{eq:torus-box-lower},
\begin{equation}
\label{eq:explicit-amplifier-constant}
c(P,3/100)=\beta|\det R_P|(2\rho_{\rm box})^4e^{-1}
>(2\cdot10^{-5})^4e^{-1}>e^{-50}.
\end{equation}
No other part of the amplification argument is changed.

\begin{theorem}\label{thm:explicit-upper}
For $r_i=\beta_i^{-1}~(i = 1,2)$, one has
\begin{align}
 I_{1-10^{-65}}(\nu_{r_1})&=\infty,
 &D_2(\nu_{r_1})&<1-10^{-65},\label{eq:explicit-upper-one}\\
 I_{1-10^{-75}}(\nu_{r_2})&=\infty,
 &D_2(\nu_{r_2})&<1-10^{-75}.\label{eq:explicit-upper-two}
\end{align}
More precisely, there are integers $n_1,n_2$ such that
\begin{align}
 5\le n_1<2\cdot10^{62},\qquad
 &\log\bigl(\beta_1^{n_1}Q^{(1)}_{n_1}\bigr)>21,\label{eq:explicit-collision-one}\\
 5\le n_2<2\cdot10^{72},\qquad
 &\log\bigl(\beta_2^{n_2}Q^{(2)}_{n_2}\bigr)>2371.\label{eq:explicit-collision-two}
\end{align}
\end{theorem}

The proof will be completed after constructing explicit packets for the two parameters.

\subsection{An explicit packet for the first quartic parameter}
\label{app:upper-first}

Throughout this subsection $P=P_1$, $\beta=\beta_1$, and $r=\beta^{-1}$.
We first choose an explicit reciprocal seed. Let $u_1,v_1$ be as in
\eqref{eq:trace-roots}, and put
\[
 \alpha_*=u_1+1,\qquad
 \alpha_*'=v_1+1=-\alpha_*^{-1}.
\]
Since $\alpha_*^3=13+10u_1$, define
\begin{equation*}
 K_0(X)=10X^2+13X+10.
\end{equation*}
Then
\[
 K_0(\beta)=\beta\alpha_*^3,
 \qquad
 K_0(\zeta)=\zeta(\alpha_*')^3.
\]
For a root $\lambda$ of $P$, write
$C_\lambda=P'(\lambda)^{-1}$ and define
\begin{equation*}
 \xi_i^{(0)}=-C_\beta\beta^i\alpha_*^3,
 \qquad
 \chi_i^{(0)}
 =2\Re\bigl(C_\zeta\zeta^i(\alpha_*')^3\bigr).
\end{equation*}
The reciprocity of $K_0$ gives reflection index $2$, hence
\[
 \xi_i^{(0)}
 \equiv\xi_{2-i}^{(0)}+\chi_i^{(0)}
 \pmod{\ZZ}.
\]
Set $Y_i^{(0)}=\can(\xi_i^{(0)})$ and localize by
\begin{equation*}
 H_i^{(0)}=
 \begin{cases}
  Y_i^{(0)},&i\le14,\\
  \xi_{2-i}^{(0)},&i\ge15.
 \end{cases}
\end{equation*}
Writing
\[
 P_1(X)
 =X^2\bigl((X+X^{-1})^2-(X+X^{-1})-3\bigr)
\]
gives
\[
 |P_1'(\zeta)|=\sqrt{13}\sqrt{u_1}.
\]
Thus the neutral amplitude is
\begin{equation*}
 \Delta_0
 :=2|C_\zeta K_0(\zeta)|
 =\frac{2\alpha_*^{-3}}{\sqrt{13}\sqrt{u_1}}.
\end{equation*}

\begin{lemma}\label{lem:upper-data-one}
Let
\[
 \kappa=P(\Shift)Y^{(0)},\qquad
 E_0=\sum_{i\in\ZZ}V(H_i^{(0)}).
\]
Then $\kappa$ is nonzero, integer-valued, and supported on $[-12,14]$,
with
\begin{equation*}
 \kappa_{-8}=1,\qquad
 \kappa_{-5}=-1,\qquad
 \kappa_3=-1,\qquad
 \kappa_6=1,
\end{equation*}
and all other entries zero. Moreover,
\begin{equation*}
 E_0<\frac{61}{5},
 \qquad
 \frac12-|H_i^{(0)}|>\frac{39}{500}
 \quad(i\in\ZZ),
\end{equation*}
and
\begin{equation*}
 0<\Delta_0<\frac{203}{20000},
 \qquad
 h_0:=|\xi_{-12}^{(0)}|<\frac1{100}.
\end{equation*}
Outside the core,
\begin{equation*}
 |H_i^{(0)}|
 \le h_0r^{\dist(i,[-12,14])}
 <\frac1{100}r^{\dist(i,[-12,14])}.
\end{equation*}
Finally, the even integers
\begin{equation}\label{eq:upper-one-delays}
 L_1=124,\qquad
 L_2=489637838,\qquad
 L_3=313586165619659333154062
\end{equation}
satisfy
\begin{equation}\label{eq:upper-one-returns}
 |1-\zeta^{L_1}|<\frac{2507}{10^6},\qquad
 |1-\zeta^{L_2}|<\frac{3736}{10^{12}},\qquad
 |1-\zeta^{L_3}|<\frac{5873}{10^{27}}.
\end{equation}
\end{lemma}

The finite assertions in Lemma~\ref{lem:upper-data-one} are verified by
the exact rational and interval calculations described in
Appendix~\ref{app:verification}.

\begin{proposition}\label{prop:explicit-packet-one}
For $P_1$, there is a packet, in the sense of Definition~\ref{def:packet}, with
$\eta=3/100$, $h=1$, and
\begin{equation*}
 \begin{gathered}
 \ell=313586165619659822792051<4\cdot10^{23},\\
 E<105,\qquad
 \mathcal V<1.4\cdot10^{25},\qquad
 0<\Delta<6\cdot10^{-37}.
 \end{gathered}
\end{equation*}
The associated forcing is nonzero, integer-valued, and supported on the core,
and the localized sequence satisfies
\[
 |H_i|<\frac12-\frac3{50}
 \qquad(i\in\ZZ).
\]
\end{proposition}
\begin{proof}
Starting from the seed in Lemma~\ref{lem:upper-data-one},
apply Lemma~\ref{lem:mirror} successively with the three delays in
\eqref{eq:upper-one-delays}. The exact rational comparisons described in
Appendix~\ref{app:verification} verify all hypotheses of the three
applications and give
\[
 E<105,\qquad
 \mathcal V<1.4\cdot10^{25},\qquad
 0<\Delta<6\cdot10^{-37},
\]
with localized boundary distance greater than $779/10000>3/50$.
The resulting reciprocal polynomial is
\[
 K_3(X)=(10X^2+13X+10)\prod_{j=1}^3(1-X^{-L_j}),
\]
and its core length is
\[
 27+L_1+L_2+L_3
 =313586165619659822792051
 <4\cdot10^{23}.
\]
Lemma~\ref{lem:mirror} preserves nonzero integral core-supported forcing
and does not increase the real-tail coefficient beyond $1/64$.
After translating the core to $[0,\ell-1]$, we therefore obtain the
claimed packet; the coarser choice $h=1$ is admissible.
\end{proof}

\subsection{An explicit packet for the second quartic parameter}
\label{app:upper-second}

Now $P=P_2$, $\beta=\beta_2$, and $r=\beta^{-1}$. For a root $\lambda$,
put $C_\lambda=P'(\lambda)^{-1}$, and let
$e^{(0)}=(e_i^{(0)})_{i\in\ZZ}$ be defined by
\[
 e_i^{(0)}=
 \begin{cases}
  1,&i=0,\\
  0,&i\ne0.
 \end{cases}
\]
We begin by constructing an explicit finite seed.

\begin{lemma}\label{lem:green-sequence}
The real sequence
\begin{equation*}
 g_i=
 \begin{cases}
  -C_\beta\beta^{i-1},&i\le0,\\
  C_rr^{i-1}+2\Re(C_\zeta\zeta^{i-1}),&i\ge1
 \end{cases}
\end{equation*}
satisfies
\[
 P(\Shift)g=e^{(0)}.
\]
\end{lemma}

\begin{proof}
Both branches satisfy the homogeneous $P$ recurrence. Lagrange
interpolation gives
\[
 \sum_{P(\lambda)=0}\frac{\lambda^k}{P'(\lambda)}
 =0\quad(k=0,1,2),
 \qquad
 \sum_{P(\lambda)=0}\frac{\lambda^3}{P'(\lambda)}=1.
\]
Substitution at the four indices $-3,-2,-1,0$, where the recurrence
crosses the junction between the two branches, gives respectively
$0,0,0,1$. Hence $P(\Shift)g=e^{(0)}$.
\end{proof}

Define
\begin{equation*}
 \begin{aligned}
 K_{\rm base}(X)
 &=1-X+X^2-X^3+X^4-X^7+X^8-X^9+X^{10}-X^{11},\\
 K_{\rm seed}(X)
 &=K_{\rm base}(X)(1-X^{2673})
 =\sum_{j=0}^{2684} k_jX^j,
 \end{aligned}
\end{equation*}
and set
\begin{equation}\label{eq:upper-two-seed}
 Y_i^{(1)}=\sum_{j=0}^{2684} k_jg_{i-100-j},
 \qquad
 \ell_1=2885.
\end{equation}
By Lemma~\ref{lem:green-sequence}, the forcing is
$(k_{i-100})_i$, a nonzero integer sequence supported on
$[0,\ell_1-1]$. Outside the core,
\begin{equation}\label{eq:upper-two-tails}
 \begin{aligned}
 Y_i^{(1)}
 &=a_-^{(1)}\beta^i
 &&(i<0),\\
 Y_i^{(1)}
 &=2\Re(\gamma_1\zeta^i)
   +a_+^{(1)}r^{i-\ell_1}
 &&(i\ge\ell_1),
 \end{aligned}
\end{equation}
where
\begin{equation*}
 \begin{aligned}
 a_-^{(1)}
 &=-C_\beta r^{101}\sum_{j=0}^{2684} k_j r^j,\\
 a_+^{(1)}
 &=C_r\sum_{j=0}^{2684} k_j r^{\ell_1-101-j},\\
 \gamma_1
 &=C_\zeta\zeta^{-101}\sum_{j=0}^{2684} k_j\zeta^{-j}.
 \end{aligned}
\end{equation*}

\begin{lemma}\label{lem:upper-seed-two}
The seed in \eqref{eq:upper-two-seed} satisfies
\begin{equation*}
 |a_-^{(1)}|,|a_+^{(1)}|<10^{-24},
 \qquad
 0<\Delta_1:=2|\gamma_1|<\frac{47}{10^7},
\end{equation*}
and
\begin{equation*}
 \max_{0\le i<\ell_1}|Y_i^{(1)}|<0.394,
 \qquad
 E_1:=\sum_{i=0}^{\ell_1-1}V(Y_i^{(1)})<24.
\end{equation*}
\end{lemma}

The finite assertions in Lemma~\ref{lem:upper-seed-two} are verified by
the exact rational and interval calculations described in
Appendix~\ref{app:verification}.

The next step uses the same delayed-difference operation as
Lemma~\ref{lem:mirror}. For the explicit seed above, we use the following
quantitative form based on the tail decomposition
\eqref{eq:upper-two-tails}.
For $m\ge0$, set
\begin{equation*}
 a_r(m)=
 \left(\frac{1-r^{2m}}{1-r^2}\right)^{1/2},
 \qquad
 a_\zeta(m)=
 \left(
 \min\left\{
 m,\frac m2+\frac1{|1-\zeta^2|}
 \right\}
 \right)^{1/2}.
\end{equation*}

\begin{lemma}\label{lem:delayed-difference}
Let $Y$ be a canonical recurrence sequence with core $[0,\ell-1]$,
core action $E$, core maximum $a_*$, and tails
\[
 Y_i=a_-\beta^i\quad(i<0),\qquad
 Y_i=2\Re(\gamma\zeta^i)+a_+r^{i-\ell}\quad(i\ge\ell),
\]
where $|a_-|,|a_+|\le\tau$ and $\Delta=2|\gamma|>0$.
For $L\ge\ell$, put $\ell_{\rm gap}=L-\ell$ and
\[
 Y_i^+=Y_i-Y_{i-L}.
\]
Set
\begin{equation*}
\begin{aligned}
 \alpha_1
 &=\tau r^{\ell_{\rm gap}+1}a_r(\ell),\\
 \alpha_2
 &=\Delta a_\zeta(\ell)
   +\tau r^{\ell_{\rm gap}}a_r(\ell),\\
 \alpha_{\rm gap}
 &=\Delta a_\zeta(\ell_{\rm gap})
   +\tau(1+r)a_r(\ell_{\rm gap}),\\[2pt]
 m_c
 &=a_*+\Delta+\tau r^{\ell_{\rm gap}},
 \qquad
 m_{\rm gap}=2\Delta+2\tau.
\end{aligned}
\end{equation*}
If $m_c,m_{\rm gap}<1/2$, then $Y^+$ is canonical, with nonzero
integer-valued forcing supported on $[0,\ell+L-1]$, and its real-tail
coefficients remain bounded by $\tau$. Its neutral amplitude is
\begin{equation}\label{eq:delayed-amplitude}
 \Delta^+=\Delta|1-\zeta^{-L}|,
\end{equation}
while its core action satisfies
\begin{equation}\label{eq:delayed-action}
 E^+\le
 2E+2\sqrt{C(m_c)E}(\alpha_1+\alpha_2)
 +C(m_c)(\alpha_1^2+\alpha_2^2)
 +C(m_{\rm gap})\alpha_{\rm gap}^2,
\end{equation}
where $C(t)=\pi^2\sec^2(\pi t)$.
\end{lemma}

\begin{proof}
The proof follows the same core--gap decomposition as
Lemma~\ref{lem:mirror}. We record only the estimates that differ.
The forcing is the old forcing minus its translate by $L$; since the
two supports are disjoint, it remains nonzero and integer-valued.
The real-tail coefficients are multiplied by factors of modulus at most
one, while the neutral coefficient is multiplied by $1-\zeta^{-L}$,
giving \eqref{eq:delayed-amplitude}.

On the first core, the gap, and the translated core, the perturbations
have $\ell^2$ norms bounded respectively by
$\alpha_1,\alpha_{\rm gap},\alpha_2$. In the neutral direction this uses
\[
 \sum_{j=0}^{m-1}\cos^2(\theta+j\varphi)
 \le
 \min\left\{m,\frac m2+\frac1{|1-\zeta^2|}\right\}.
\]
The corresponding pointwise bounds are $m_c$ and $m_{\rm gap}$, so
$Y^+$ remains canonical. Finally,
\[
 V''\le2C(t),\qquad |V'(u)|^2\le4C(t)V(u)
 \quad\text{on }[-t,t],
\]
and Taylor's theorem together with Cauchy--Schwarz gives
\eqref{eq:delayed-action}.
\end{proof}

\begin{proposition}\label{prop:explicit-packet-two}
For $P_2$, there exists a packet, in the sense of
Definition~\ref{def:packet}, with parameters
$\eta=3/100$ and $h=10^{-23}$, satisfying
\begin{equation*}
 \begin{gathered}
 \ell=34250431223905123167647850997<10^{29},\\
 E<100,\qquad
 \mathcal V<10^{18},\qquad
 0<\Delta<3.2\cdot10^{-44}.
 \end{gathered}
\end{equation*}
The associated forcing is nonzero, integer-valued, and supported on the
core, and the localized sequence satisfies
\[
 |H_i|<\frac25
 \qquad(i\in\ZZ).
\]
\end{proposition}

\begin{proof}
Use the two delays
\begin{equation*}
 \widetilde L_1=1846082221,
 \qquad
 \widetilde L_2=34250431223905123165801765891.
\end{equation*}
Repeated squaring with outward rectangular intervals gives
\begin{equation}\label{eq:upper-two-returns}
 |1-\zeta^{\widetilde L_1}|<\frac{2365}{10^{13}},
 \qquad
 |1-\zeta^{\widetilde L_2}|<\frac{2876}{10^{32}}.
\end{equation}
Define
\[
 Y_i^{(2)}=Y_i^{(1)}-Y_{i-\widetilde L_1}^{(1)},
 \qquad
 \ell_2=\ell_1+\widetilde L_1,
\]
and
\[
 Y_i^{(3)}=Y_i^{(2)}-Y_{i-\widetilde L_2}^{(2)},
 \qquad
 \ell=\ell_2+\widetilde L_2.
\]
Using the seed bounds in Lemma~\ref{lem:upper-seed-two},
apply Lemma~\ref{lem:delayed-difference} twice with $\tau=10^{-24}$.
The exact rational and interval comparisons described in
Appendix~\ref{app:verification} verify the hypotheses of both applications
and give
\[
 \ell_2=1846085106,\qquad
 E_2<49,\qquad
 \Delta_2<1.112\cdot10^{-15},
\]
and
\[
 \max_{0\le i<\ell}|Y_i^{(3)}|<0.395,
 \qquad
 E_3<99,
 \qquad
 0<\Delta_3<3.2\cdot10^{-44}.
\]
The two real-tail coefficients remain below $10^{-24}$.

Removing the outgoing neutral term from the right tail gives the
localized sequence $H$. Since $r>1/2$,
\[
 |H_i|
 \le
 10^{-23}r^{\dist(i,[0,\ell-1])}
 \qquad(i\notin[0,\ell-1]).
\]
Together with the core bound above, this gives
\[
 |H_i|<\frac25
 \qquad(i\in\ZZ).
\]
By \eqref{eq:potential-tail}, the two localized tails contribute less
than one to the action, and hence $E<100$. Moreover, on
$[-2/5,2/5]$,
\[
 |V'(t)|\le500|t|,
 \qquad
 V(t)\ge\pi^2t^2.
\]
Thus Cauchy--Schwarz, followed by the geometric tail bound, gives
\[
 \mathcal V
 \le
 \frac{500}{\pi}\sqrt{\ell E_3}
 +\frac{1000\cdot10^{-23}r}{1-r}
 <10^{18}.
\]
Finally,
\[
 \ell
 =34250431223905123167647850997
 <10^{29}.
\]
Lemma~\ref{lem:delayed-difference} preserves the nonzero
integer-valued forcing and its support on the core, completing the
construction.
\end{proof}

\subsection{Completion of the upper-bound proof}
\label{app:upper-completion}

Set
\begin{equation*}
 N_1=10^{60},\qquad N_2=10^{70},\qquad M=100.
\end{equation*}
Propositions~\ref{prop:explicit-packet-one} and
\ref{prop:explicit-packet-two}, together with
\eqref{eq:placement-blocks}--\eqref{eq:amplifier-gain}, give the following
bounds.

\begin{table}[htbp]
 \caption{Explicit inputs to Proposition~\ref{prop:packet-amplification},
 with $\eta=3/100$.}
 \label{tab:explicit-amplifier}
 \centering
 \begin{tabular}{lcc}
 \toprule
 Quantity&$P_1$&$P_2$\\
 \midrule
 $u_{N_i}$&$<5\cdot10^{-36}$&$<2.6\cdot10^{-43}$\\
 $C_{N_i}$&$<105.001$&$<100.001$\\
 $b_{N_i}$&$>10^{58}$&$>10^{68}$\\
 $\Theta_{N_i}$&$>10$&$>35$\\
 \bottomrule
 \end{tabular}
\end{table}

The strict scalar comparisons in Table~\ref{tab:explicit-amplifier} are
verified by the exact calculations described in
Appendix~\ref{app:verification}.

\begin{proof}[Proof of Theorem~\ref{thm:explicit-upper}]
Apply Proposition~\ref{prop:packet-amplification} with $M=100$ and the
constant in \eqref{eq:explicit-amplifier-constant}. If $N=10^k$, then
$b_N<N$, and hence
\[
 \log(\beta^nQ_n)
 >M\Theta_N-6k\log10-5\log100-4\log2-50.
\]
Using Table~\ref{tab:explicit-amplifier}, together with
$2<\log10<5/2$ and $1/2<\log2<1$, gives integers $n_1,n_2$ such that
\[
 5\le n_1<2\cdot10^{62},
 \qquad
 \log\bigl(\beta_1^{n_1}Q_{n_1}^{(1)}\bigr)>21,
\]
and
\[
 5\le n_2<2\cdot10^{72},
 \qquad
 \log\bigl(\beta_2^{n_2}Q_{n_2}^{(2)}\bigr)>2371.
\]
This proves
\eqref{eq:explicit-collision-one}--\eqref{eq:explicit-collision-two}.

Set
\[
 s_1^+=1-10^{-65},
 \qquad
 s_2^+=1-10^{-75}.
\]
Since $\beta_i<2$, $\log2<1$, and the bounds on $n_i$ above hold,
\[
 (1-s_i^+)n_i\log\beta_i<\frac1{500}
 \qquad(i=1,2).
\]
Therefore
\[
 \log\bigl(\beta_i^{s_i^+n_i}Q_{n_i}^{(i)}\bigr)>0
 \qquad(i=1,2).
\]
Corollary~\ref{cor:finite-collision} gives
\[
 I_{s_i^+}(\nu_{r_i})=\infty
\]
and
\[
 D_2(\nu_{r_i})
 \le-\frac{\log Q_{n_i}^{(i)}}{n_i\log\beta_i}
 <s_i^+.
\]
Thus \eqref{eq:explicit-upper-one}--\eqref{eq:explicit-upper-two} follow.
\end{proof}

\section{A finite criterion for lower bounds}\label{app:lower-criterion}

The lower bounds use the same collision probabilities $Q_n$ as the upper
bounds, but require an upper bound for the complete return series
\[
 G(t)=\sum_{n=0}^\infty t^nQ_n.
\]
We reduce this to a supersolution inequality on a finite state set. The
explicit supersolutions are given in Appendix~\ref{app:lower}, and their
exact verification is described in Appendix~\ref{app:verification}.

Let
\[
 P(X)=X^4-AX^3+BX^2-AX+1
\]
be one of the two quartics. On $\ZZ^4$, define the companion map
\begin{equation}\label{eq:affine-companion-map}
 \mathsf M_P(x_0,x_1,x_2,x_3)
 =\bigl(x_1,x_2,x_3,-x_0+Ax_1-Bx_2+Ax_3\bigr),
\end{equation}
and put $e=(0,0,0,1)$. Consider the affine chain
\begin{equation*}
 Z_{n+1}=\mathsf M_PZ_n+U_ne,
 \qquad Z_0=0,
\end{equation*}
where the digits have the law in \eqref{eq:lazy-return}.
This is an
integral-coordinate realization of the same difference process.

For $y\in\ZZ^4$, let $\PP_y$ denote the law of the same chain with
initial state $Z_0=y$.

For every root $z$ of $P$, define
\begin{equation*}
 \mathcal L_z(x)
 =(z^3-Az^2+Bz-A)x_0+(z^2-Az+B)x_1+(z-A)x_2+x_3.
\end{equation*}
Reduction modulo $P$ gives
\begin{equation}\label{eq:eigen-coordinate-dynamics}
 \mathcal L_z(\mathsf M_Px+ue)=z\mathcal L_z(x)+u,
 \qquad u\in\ZZ.
\end{equation}
For $z=\beta$, the map $\mathcal L_\beta:\ZZ^4\to\RR$ is injective,
since no nonzero integer polynomial of degree at most three can vanish
at the degree-four number $\beta$.

\begin{lemma}\label{lem:endpoint-domination}
For every $n\ge0$,
\begin{equation}\label{eq:complete-return-identity}
 \PP_0(Z_n=0)=Q_n.
\end{equation}
Moreover, for every $m\ge0$ and every $x,y\in\ZZ^4$,
\begin{equation*}
 \PP_0(Z_m=x)\le Q_m,
 \qquad
 \PP_y(Z_m=0)\le Q_m.
\end{equation*}
\end{lemma}

\begin{proof}
Iterating \eqref{eq:eigen-coordinate-dynamics} gives
\[
 \mathcal L_\beta(Z_n)
 =\sum_{j=0}^{n-1}U_j\beta^{n-1-j}.
\]
Injectivity of $\mathcal L_\beta$ and reversal of the independent digits
give \eqref{eq:complete-return-identity}.

Let $\mu_m$ be the endpoint law from the origin. Its characteristic
function is
\[
 \widehat\mu_m(\theta)
 =\prod_{j=0}^{m-1}
 \left(
 \frac12+\frac12
 \cos\bigl(2\pi\langle\mathsf M_P^je,\theta\rangle\bigr)
 \right),
\]
which is nonnegative. Hence Fourier inversion gives
\[
 \mu_m(x)
 \le\int_{\TT^4}\widehat\mu_m(\theta)\,d\theta
 =\mu_m(0)=Q_m.
\]
Starting from $y$ translates the endpoint distribution by
$\mathsf M_P^my$, proving the second inequality.
\end{proof}

The expanding coordinate separates states from which return is
impossible. Set
\begin{equation*}
 d(x)=(\beta-1)\mathcal L_\beta(x),
 \qquad
 \Omega=\{x\in\ZZ^4:|d(x)|<1\}.
\end{equation*}
If $|d(x)|\ge1$ and $|u|\le1$, then
\[
 |d(\mathsf M_Px+ue)|
 \ge\beta|d(x)|-(\beta-1)\ge1.
\]
Thus $\ZZ^4\setminus\Omega$ is forward invariant, and no state outside
$\Omega$ can return to the origin.

Fix $J\ge1$ and put
\begin{equation*}
 \Omega_J
 =\{x\in\Omega:0<\|x\|_\infty\le J\}.
\end{equation*}
Starting from the origin, let $\tau_J$ be the first positive time at
which the chain returns to $0$ or leaves $\Omega_J$. A return to $0$ is
kept, as is an exit from the cube that remains in $\Omega$; an exit from
$\Omega$ is discarded. Define
\begin{equation*}
 \pi_k^{(J)}
 =\PP_0\bigl(\tau_J=k,\ \text{the first stop is kept}\bigr),
 \qquad
 \Phi_J(t)=\sum_{k=1}^{\infty}\pi_k^{(J)}t^k.
\end{equation*}

\begin{proposition}\label{prop:first-stop}
For every $n\ge1$,
\begin{equation}\label{eq:first-stop-domination}
 Q_n\le\sum_{k=1}^n\pi_k^{(J)}Q_{n-k}.
\end{equation}
Consequently, if $t>0$ and
\begin{equation*}
 \Phi_J(t)\le c
\end{equation*}
for some $0\le c<1$, then
\begin{equation}\label{eq:return-series-from-first-stop}
 G(t)\le\frac1{1-c}.
\end{equation}

\end{proposition}

\begin{proof}
A path returning to the origin at time $n$ must have a kept first stop
at some time $k\le n$. Conditional on its endpoint, the probability of
returning in the remaining $n-k$ steps is at most $Q_{n-k}$ by
Lemma~\ref{lem:endpoint-domination}. This gives
\eqref{eq:first-stop-domination}.

For
\[
 G_N(t)=\sum_{n=0}^Nt^nQ_n,
\]
multiplication by $t^n$ and summation up to $N$ give
\[
 G_N(t)
 \le1+\sum_{k=1}^N\pi_k^{(J)}t^kG_{N-k}(t)
 \le1+cG_{N-1}(t).
\]
Induction yields $G_N(t)\le(1-c)^{-1}$, and letting $N\to\infty$ proves
\eqref{eq:return-series-from-first-stop}.
\end{proof}

We now bound the first-stop transform by finitely many inequalities.
Let $\mathcal F\subset\Omega_J$ be finite, symmetric under
$x\mapsto-x$, with $e,-e\in\mathcal F$, and assume that
\[
 x\in\mathcal F,\quad
 u\in\{-1,0,1\},\quad
 \mathsf M_Px+ue\in\Omega_J
 \quad\Longrightarrow\quad
 \mathsf M_Px+ue\in\mathcal F.
\]
For a successor $y=\mathsf M_Px+ue$ of a state $x\in\mathcal F$,
define
\begin{equation*}
 \widetilde w(y)=
 \begin{cases}
  w(y),&y\in\mathcal F,\\
  1,&y=0\ \text{or}\ \bigl(y\in\Omega,\ \|y\|_\infty>J\bigr),\\
  0,&y\notin\Omega.
 \end{cases}
\end{equation*}
By the closure assumption, these are the only cases that can occur.

\begin{theorem}\label{thm:finite-supersolution}
Let $t>0$. Suppose that $w:\mathcal F\to[0,\infty)$ is even and satisfies
\begin{equation}\label{eq:supersolution}
 w(x)\ge\frac t4\Bigl(
  \widetilde w(\mathsf M_Px-e)
  +2\widetilde w(\mathsf M_Px)
  +\widetilde w(\mathsf M_Px+e)
 \Bigr)
 \qquad(x\in\mathcal F).
\end{equation}
If
\begin{equation}\label{eq:supersolution-initial-bound}
 \bar c=\frac t2\bigl(1+w(e)\bigr)<1,
\end{equation}
then
\begin{equation}\label{eq:supersolution-green-bound}
 G(t)\le\frac1{1-\bar c}.
\end{equation}
\end{theorem}

\begin{proof}
For $x\in\mathcal F$, let $\phi_m(x)$ be the $t$-weighted mass of kept
first stops occurring within at most $m$ further steps, with
$\phi_0=0$. The one-step recursion uses $\phi_m$ at internal successors,
the terminal value $1$ at a return or a return-capable exit, and $0$ at
an exit from $\Omega$. By \eqref{eq:supersolution} and induction,
\[
 0\le\phi_m(x)\le w(x)
 \qquad(x\in\mathcal F).
\]
Passing to the monotone limit bounds the complete first-stop transform
from every internal state by $w$.

From the origin, the zero digit returns immediately, while the other two
digits move to $e$ and $-e$. Hence, by evenness,
\[
 \Phi_J(t)
 \le\frac t2+\frac t4w(e)+\frac t4w(-e)
 =\frac t2\bigl(1+w(e)\bigr)
 =\bar c.
\]
Proposition~\ref{prop:first-stop} gives
\eqref{eq:supersolution-green-bound}.
\end{proof}

Combining Theorem~\ref{thm:finite-supersolution} with
Corollary~\ref{cor:finite-collision}, any choice $1<t<\beta$ for which
\eqref{eq:supersolution-initial-bound} holds gives
\[
 D_2(\nu_r)\ge\frac{\log t}{\log\beta},
\]
and
\[
 I_s(\nu_r)<\infty
 \qquad
 \left(0<s<\frac{\log t}{\log\beta}\right).
\]
Appendix~\ref{app:lower} applies this criterion at one rational value of
$t$ for each quartic parameter.

\section{Explicit lower bounds}\label{app:lower}

For the two quartic parameters, put
\begin{equation*}
 t_1=\frac{68883}{40000},
 \qquad
 t_2=\frac{235399}{125000},
 \qquad
 s_i=\frac{\log t_i}{\log\beta_i}\quad(i=1,2).
\end{equation*}
We apply the finite criterion of Appendix~\ref{app:lower-criterion} at
these two values of $t$. The corresponding state lists and integer
supersolutions are contained in the verification release; their exact
verification is described in Appendix~\ref{app:verification}.

\begin{theorem}\label{thm:explicit-lower}
With the indexed notation fixed in Section~\ref{sec:collisions}, one has
\begin{equation}\label{eq:explicit-green-bounds}
 G_1(t_1)<105,
 \qquad
 G_2(t_2)<624.
\end{equation}
Consequently,
\begin{equation}\label{eq:explicit-lower-energy}
 I_s(\nu_{r_i})<\infty\quad(0<s<s_i),
 \qquad
 D_2(\nu_{r_i})\ge s_i
 \quad(i=1,2).
\end{equation}
The lower exponents satisfy
\begin{align}
 0.9999905922652119&<s_1<0.9999905922652120,
 \label{eq:lower-exponent-one}\\
 0.9999903475160519&<s_2<0.9999903475160520.
 \label{eq:lower-exponent-two}
\end{align}
In particular,
\begin{equation}\label{eq:readable-lower-bounds}
 D_2(\nu_{r_1})>0.99999,
 \qquad
 D_2(\nu_{r_2})>0.99999.
\end{equation}
\end{theorem}

We now record the finite supersolutions used in the theorem. For a
nonzero vector $x\in\ZZ^4$, let $\can_{\pm}(x)$ be the member of
$\{x,-x\}$ whose first nonzero coordinate is positive. For $i=1,2$, the
state list consists of distinct canonical representatives
\[
 x_{i,0},\ldots,x_{i,m_i-1},
 \qquad
 x_{i,0}=e,
\]
in $\Omega_{J_i}$, where
\begin{equation*}
 (J_1,m_1)=(180,455137),
 \qquad
 (J_2,m_2)=(240,642451).
\end{equation*}
Set
\[
 \mathcal F_i=\{\pm x_{i,j}:0\le j<m_i\}.
\]
Every successor that remains in $\Omega_{J_i}$ is represented in
$\mathcal F_i$; a return to zero or an exit from the cube that remains
in $\Omega$ has terminal value $1$, while an exit from $\Omega$ has
terminal value $0$.

Let
\begin{equation*}
 \Lambda=2^{80}.
\end{equation*}
The second finite array consists of nonnegative integers
$\widehat w_{i,j}$, and we set
\[
 w_i(\pm x_{i,j})
 =\frac{\widehat w_{i,j}}{\Lambda}.
\]
Write $t_i=a_i/b_i$ in lowest terms, so that
\[
 (a_1,b_1)=(68883,40000),
 \qquad
 (a_2,b_2)=(235399,125000).
\]
For a successor $y$ of a listed state, let
$\widehat{\widetilde w}_i(y)$ denote the numerator of its terminally
extended value: it is $\widehat w_{i,k}$ for an internal successor
represented by $x_{i,k}$, $\Lambda$ for a return or a return-capable
exit, and $0$ for an exit from $\Omega$. Define
\begin{equation}\label{eq:integer-row-residual}
 \mathcal R_{i,j}
 =4b_i\widehat w_{i,j}
 -a_i\Bigl(
 \widehat{\widetilde w}_i(\mathsf M_{P_i}x_{i,j}-e)
 +2\widehat{\widetilde w}_i(\mathsf M_{P_i}x_{i,j})
 +\widehat{\widetilde w}_i(\mathsf M_{P_i}x_{i,j}+e)
 \Bigr).
\end{equation}
Then $\mathcal R_{i,j}\ge0$ is exactly the supersolution inequality
\eqref{eq:supersolution} after clearing denominators.

\begin{proposition}\label{prop:integer-supersolutions}
The sets $\mathcal F_i$ are sign-symmetric, contain $\pm e$, and are
closed under every successor that remains in $\Omega_{J_i}$. The integer
arrays are nonnegative and satisfy
\begin{align*}
 \min_{0\le j<m_1}\mathcal R_{1,j}
 &=2791379258189148576>0,
 \\
 \min_{0\le j<m_2}\mathcal R_{2,j}
 &=8688248611889750432>0.
\end{align*}
Their values at $e$ are
\begin{align*}
 \widehat w_{1,0}
 &=181660711014187729550784,
 \\
 \widehat w_{2,0}
 &=72924590283342799761072.
\end{align*}
If
\begin{equation*}
 \bar c_i
 =\frac{t_i}{2}
 \left(1+\frac{\widehat w_{i,0}}{\Lambda}\right),
\end{equation*}
then
\begin{equation}\label{eq:lower-initial-gaps}
 1-\bar c_1>\frac1{105},
 \qquad
 1-\bar c_2>\frac1{624}.
\end{equation}
Hence $w_i:\mathcal F_i\to[0,\infty)$ satisfies the hypotheses of
Theorem~\ref{thm:finite-supersolution} for $P_i$, $J_i$, and $t_i$.
\end{proposition}

The finite assertions in Proposition~\ref{prop:integer-supersolutions}
are verified by the exact integer and rational calculations described in
Appendix~\ref{app:verification}.

\begin{proof}[Proof of Theorem~\ref{thm:explicit-lower}]
Apply Theorem~\ref{thm:finite-supersolution} with the supersolutions in
Proposition~\ref{prop:integer-supersolutions}. By
\eqref{eq:lower-initial-gaps},
\[
 G_1(t_1)
 \le\frac1{1-\bar c_1}<105,
 \qquad
 G_2(t_2)
 \le\frac1{1-\bar c_2}<624,
\]
which proves \eqref{eq:explicit-green-bounds}. The exact root comparisons
in Appendix~\ref{app:verification} give
\[
 1<t_i<\beta_i
 \qquad(i=1,2).
\]
Corollary~\ref{cor:finite-collision} therefore yields
\eqref{eq:explicit-lower-energy}.

Finally, the rational logarithm bounds described in
Appendix~\ref{app:verification} give
\eqref{eq:lower-exponent-one}--\eqref{eq:lower-exponent-two}. In
particular, $s_i>0.99999$, proving
\eqref{eq:readable-lower-bounds}. No assertion is made about finiteness
of the energy at $s=s_i$.
\end{proof}

Combining Theorem~\ref{thm:explicit-lower} with
Theorem~\ref{thm:explicit-upper} gives the bounds stated in
Theorem~\ref{thm:quantitative}. The intervals
\eqref{eq:lower-exponent-one}--\eqref{eq:lower-exponent-two} enclose the
lower exponents $s_i$, not the dimensions themselves.

\section{Exact arithmetic and verification}\label{app:verification}

The explicit inequalities in Appendices~\ref{app:upper} and
\ref{app:lower} are checked by exact integer arithmetic or by rational
intervals rounded outwards. Numerical searches may be used to locate
candidate packets or supersolutions, but no floating-point comparison is
used in their acceptance.

For the upper bounds, all real interval endpoints lie on the grid with
denominator $10^{200}$. Addition is exact on this grid; products and
quotients are rounded outwards by integer floor and ceiling division,
and division is allowed only when the denominator interval has been
proved not to meet zero. Square roots are enclosed by integer-square
comparisons. Complex quantities are represented by rectangular interval enclosures, and large powers of $\zeta$ are evaluated by repeated
squaring. In particular, the return inequalities in
\eqref{eq:upper-one-returns} and \eqref{eq:upper-two-returns} are checked
directly from the algebraic description of $\zeta$. Centering is accepted
only when the corresponding interval determines a unique integer and
remains strictly inside $(-1/2,1/2)$ after subtraction.

The potential
\[
 V(x)=-2\log\cos(\pi x)
\]
is enclosed without floating-point transcendental evaluations. Machin's
identity
\[
 \pi=16\arctan(1/5)-4\arctan(1/239)
\]
is combined with alternating-series bounds for the two arctangents,
while sine and cosine are evaluated by Taylor polynomials with explicit
remainders. For $0<c\le1$, putting $z=(1-c)/(1+c)$ gives
\begin{equation}\label{eq:verification-log-cosine}
 -2\log c
 =4\sum_{j=0}^{m-1}\frac{z^{2j+1}}{2j+1}+R_m,
 \qquad
 0\le R_m\le
 \frac{4z^{2m+1}}{(2m+1)(1-z^2)}.
\end{equation}
The verifier uses $160$ terms in each arctangent series, $100$ terms in
each sine and cosine series, and $m=160$ in
\eqref{eq:verification-log-cosine}. For sine and cosine the arguments lie
in $[-2,2]$, with remainder bounds $|x|^{201}/201!$ and
$|x|^{200}/200!$, respectively. These outward bounds verify the seed
actions, boundary distances, return estimates, mirror and
delayed-difference bounds, derivative sums, and the scalar inequalities
in Table~\ref{tab:explicit-amplifier}. The final large cores are not
expanded coordinate by coordinate; their bounds follow from the analytic
lemmas applied to the verified finite seeds.

For the lower bounds, each real root $\beta_i$ is first isolated by exact
rational bisection, starting from
\[
 \frac{17}{10}<\beta_1<\frac74,
 \qquad
 \frac{15}{8}<\beta_2<\frac{19}{10}.
\]
The signs of the defining polynomials at the endpoints, together with a
rational lower bound for the derivative, show that each interval contains
exactly one root. Two hundred bisections give strict rational brackets
\[
 \beta_i^-<\beta_i<\beta_i^+,
\]
which are used in all later root-dependent comparisons. In particular,
$1<t_i<\beta_i$ is proved by comparison with the rational lower endpoint
$\beta_i^-$.

Membership in the return strip $\Omega$ is also decided without a
decimal approximation to $\beta_i$. Reducing
$(X-1)\mathcal L_X(x)$ modulo
\[
 P(X)=X^4-AX^3+BX^2-AX+1
\]
gives
\begin{equation}\label{eq:verification-strip-polynomials}
 d(x)=\sum_{j=0}^3x_jd_j(\beta),
\end{equation}
where
\[
 \begin{aligned}
 d_0(X)&=-X^3+AX^2-BX+A-1,\\
 d_1(X)&=X^3-(A+1)X^2+(A+B)X-B,\\
 d_2(X)&=X^2-(A+1)X+A,\\
 d_3(X)&=X-1.
 \end{aligned}
\]
Interval Horner evaluation on $[\beta_i^-,\beta_i^+]$ gives outward
rational enclosures
\[
 \frac{\underline d_j}{\Lambda}
 \le d_j(\beta_i)\le
 \frac{\overline d_j}{\Lambda},
 \qquad
 \Lambda=2^{80}.
\]
For $x\in\ZZ^4$, define
\[
 \begin{aligned}
 d_-(x)
 &=\sum_{x_j\ge0}x_j\underline d_j
   +\sum_{x_j<0}x_j\overline d_j,\\
 d_+(x)
 &=\sum_{x_j\ge0}x_j\overline d_j
   +\sum_{x_j<0}x_j\underline d_j.
 \end{aligned}
\]
Then
\begin{equation*}
 -\Lambda<d_-(x),\qquad d_+(x)<\Lambda
 \quad\Longrightarrow\quad x\in\Omega,
\end{equation*}
whereas
\begin{equation*}
 d_-(x)\ge\Lambda
 \quad\text{or}\quad
 d_+(x)\le-\Lambda
 \quad\Longrightarrow\quad x\notin\Omega.
\end{equation*}
If neither interval test decides membership in $\Omega$, the polynomial in
\eqref{eq:verification-strip-polynomials} is examined exactly. A constant
case is compared with $\pm1$ directly; an unresolved nonconstant case is
rejected. Since a nonconstant integer polynomial of degree at most three
cannot take the rational value $\pm1$ at the degree-four number
$\beta_i$, this also handles the strict boundary $|d|=1$.

The state lists in Appendix~\ref{app:lower} are checked directly from
their coordinates. Distinctness, the canonical sign convention, the
height bound, strip membership, and the initial state are verified first.
For every listed state and every digit $-1,0,1$, the successor is
recomputed from \eqref{eq:affine-companion-map}. An internal successor
must occur in the state list, while returns and exits receive the terminal
values prescribed in Appendix~\ref{app:lower-criterion}. The
supersolution inequalities are then the exact integer comparisons
\[
 \begin{aligned}
 \mathcal R_{i,j}
 ={}&4b_i\widehat w_{i,j}\\
 &-a_i\Bigl(
 \widehat{\widetilde w}_i(\mathsf M_{P_i}x_{i,j}-e)
 +2\widehat{\widetilde w}_i(\mathsf M_{P_i}x_{i,j})
 +\widehat{\widetilde w}_i(\mathsf M_{P_i}x_{i,j}+e)
 \Bigr)\ge0,
 \end{aligned}
\]
namely \eqref{eq:integer-row-residual}. The same exact arithmetic verifies
the minimum residuals, the initial values $\widehat w_{i,0}$, and the
inequalities \eqref{eq:lower-initial-gaps}.

Finally, the lower exponents are compared by rational logarithm bounds.
For rational $x>1$, let $z=(x-1)/(x+1)$. Then
\begin{equation*}
 \log x
 =2\sum_{j=0}^{m-1}\frac{z^{2j+1}}{2j+1}+R_m,
 \qquad
 0<R_m\le
 \frac{2z^{2m+1}}{(2m+1)(1-z^2)}.
\end{equation*}
With $m=48$, let $\log_-(x)$ and $\log_+(x)$ denote the resulting
rational lower and upper bounds. Monotonicity of the logarithm and the
root brackets give
\[
 \frac{\log_-(t_i)}{\log_+(\beta_i^+)}
 <
 \frac{\log t_i}{\log\beta_i}
 <
 \frac{\log_+(t_i)}{\log_-(\beta_i^-)}.
\]
Clearing denominators proves the strict enclosures
\eqref{eq:lower-exponent-one}--\eqref{eq:lower-exponent-two}; in
particular, the comparison with $99999/100000$ is exact.

The verification release described in the Data and Code Availability
statement contains the finite witnesses and Python~3 programs used for
these checks. 
The upper verifier checks the finite data in
Lemmas~\ref{lem:upper-data-one} and \ref{lem:upper-seed-two}, the
majorants in Propositions~\ref{prop:explicit-packet-one} and
\ref{prop:explicit-packet-two}, and the scalar inequalities in
Table~\ref{tab:explicit-amplifier}. The lower verifier checks the
supersolutions in Proposition~\ref{prop:integer-supersolutions} and the
root and logarithm comparisons used in
Theorem~\ref{thm:explicit-lower}. 
These programs verify the finite claims used above; 
the analytic arguments are proved independently in
the paper.

\sectionnotoc{Data and Code Availability}\label{sec:data-code}

The certificate data and verification programs supporting the quantitative
bounds in Theorem~\ref{thm:quantitative} are openly available on
Zenodo~\cite{quartic_salem_verification}. The archive contains the finite witnesses, verification code, recorded
outputs, and a \texttt{README.md} file with the commands needed to
reproduce the checks and additional diagnostics.
The proof of Theorem~\ref{thm:main} is independent of this archive.

\sectionnotoc{Funding}

G.~C. was partially supported by the National Natural Science Foundation of China (NSFC), grant no.~12371126.
X.~F. was partially supported by the National Science and Technology Council, Taiwan, grant no.~114-2115-M-A49-003-MY3.

\sectionnotoc{AI Declaration}

The authors used OpenAI ChatGPT (GPT-6.0 Sol) for language editing, \LaTeX\ formatting, and limited assistance with local mathematical reasoning and numerical verification. The overall mathematical strategy and framework were developed by the authors. All mathematical content was written and checked by the authors who take full responsibility.

\sectionnotoc{Conflict of Interest}

The authors declare no conflicts of interest.

\end{document}